\documentclass[11pt]{article}
\usepackage{amsmath,amssymb,amsthm,amscd}

\newtheorem{theorem}{Theorem}[section]
\newtheorem{lemma}[theorem]{Lemma}
\newtheorem{proposition}[theorem]{Proposition}
\newtheorem{corollary}[theorem]{Corollary}

\theoremstyle{plain}

\theoremstyle{definition}
\newtheorem{definition}[theorem]{Definition}

\numberwithin{equation}{section}

\renewcommand{\labelenumi}{\textup{(\theenumi)}}

\newcommand{\BF}{\operatorname{BF}}

\newcommand{\Homeo}{\operatorname{Homeo}}

\newcommand{\id}{\operatorname{id}}
\newcommand{\Ker}{\operatorname{Ker}}

\newcommand{\Ad}{\operatorname{Ad}}

\newcommand{\K}{\mathcal{K}}
\newcommand{\C}{\mathcal{C}}

\newcommand{\N}{\mathbb{N}}
\newcommand{\Z}{\mathbb{Z}}
\newcommand{\T}{\mathbb{T}}
\newcommand{\Zp}{{\mathbb{Z}}_+}

\def\R{{\mathcal{R}}}
\def\WRA{{\widetilde{\R}_A}}
\def\WRAf{{\widetilde{\R}_{A_f}}}

\title{Flow equivalence of topological Markov shifts 
and \\
Ruelle algebras }
\author{Kengo Matsumoto \\
Department of Mathematics \\
Joetsu University of Education \\
Joetsu, 943-8512, Japan
}
\date{}

\begin{document}
\maketitle

\def\det{{{\operatorname{det}}}}

%\maketitle
\begin{abstract}
We  study discrete flow equivalence of two-sided topological Markov shifts 
by using extended Ruelle algebra.
We characterize flow equivalence of  two-sided topological Markov shifts 
in terms of conjugacy of certain actions weighted by ceiling functions
of two-dimensional torus on the stabilized extended Ruelle algebras
for the Markov shifts.
\end{abstract}

%%%%%%%%%%%%%%%%%%%%%%%%%%%%%%%%%%%%%%%%%%%%%%%%%%%%   
\def\R{{\mathcal{R}}}
\def\OA{{{\mathcal{O}}_A}}
\def\OB{{{\mathcal{O}}_B}}
\def\RA{{{\mathcal{R}}_A}}
\def\WRA{{\widetilde{\R}_A}}
\def\WRB{{\widetilde{\R}_B}}
\def\WRC{{\widetilde{\R}_C}}
\def\RAC{{{\mathcal{R}}_A^\circ}}
\def\RAR{{{\mathcal{R}}_A^\rho}}
\def\RTA{{{\mathcal{R}}_{A^t}}}
\def\RB{{{\mathcal{R}}_B}}
\def\FA{{{\mathcal{F}}_A}}
\def\FB{{{\mathcal{F}}_B}}
\def\FA{{{\mathcal{F}}_A}}
\def\FTA{{{\mathcal{F}}_{A^t}}}
\def\FB{{{\mathcal{F}}_B}}
\def\FTB{{{\mathcal{F}}_{B^t}}}
\def\DTA{{{\mathcal{D}}_{{}^t\!A}}}
\def\FDA{{{\frak{D}}_A}}
\def\FDTA{{{\frak{D}}_{{}^t\!A}}}
\def\FFA{{{\frak{F}}_A}}
\def\FFTA{{{\frak{F}}_{A^t}}}
\def\FFB{{{\frak{F}}_B}}
\def\OZ{{{\mathcal{O}}_Z}}
\def\V{{\mathcal{V}}}
\def\E{{\mathcal{E}}}
\def\G{{\mathcal{G}}}
\def\OTA{{{\mathcal{O}}_{A^t}}}
\def\OTB{{{\mathcal{O}}_{B^t}}}
\def\SOA{{{\mathcal{O}}_A}\otimes{\mathcal{K}}}
\def\SOB{{{\mathcal{O}}_B}\otimes{\mathcal{K}}}
\def\SOZ{{{\mathcal{O}}_Z}\otimes{\mathcal{K}}}
\def\SOTA{{{\mathcal{O}}_{A^t}\otimes{\mathcal{K}}}}
\def\SOTB{{{\mathcal{O}}_{B^t}\otimes{\mathcal{K}}}}
\def\DA{{{\mathcal{D}}_A}}
\def\DB{{{\mathcal{D}}_B}}
\def\DZ{{{\mathcal{D}}_Z}}

\def\SDA{{{\mathcal{D}}_A}\otimes{\mathcal{C}}}
\def\SDB{{{\mathcal{D}}_B}\otimes{\mathcal{C}}}
\def\SDZ{{{\mathcal{D}}_Z}\otimes{\mathcal{C}}}
\def\SDTA{{{\mathcal{D}}_{A^t}\otimes{\mathcal{C}}}}
\def\O2{{{\mathcal{O}}_2}}
\def\D2{{{\mathcal{D}}_2}}

%%%%%%%%%%%%%%%%%%%%%%%%%%%%%%%
\def\Max{{{\operatorname{Max}}}}
\def\Tor{{{\operatorname{Tor}}}}
\def\Ext{{{\operatorname{Ext}}}}
\def\Per{{{\operatorname{Per}}}}
\def\PerB{{{\operatorname{PerB}}}}
\def\Homeo{{{\operatorname{Homeo}}}}
\def\HA{{{\frak H}_A}}
\def\HB{{{\frak H}_B}}
\def\HSA{{H_{\sigma_A}(X_A)}}
\def\Out{{{\operatorname{Out}}}}
\def\Aut{{{\operatorname{Aut}}}}
\def\Ad{{{\operatorname{Ad}}}}
\def\Inn{{{\operatorname{Inn}}}}
\def\det{{{\operatorname{det}}}}
\def\exp{{{\operatorname{exp}}}}
\def\cobdy{{{\operatorname{cobdy}}}}
\def\Ker{{{\operatorname{Ker}}}}
\def\ind{{{\operatorname{ind}}}}
\def\id{{{\operatorname{id}}}}
\def\supp{{{\operatorname{supp}}}}
\def\co{{{\operatorname{co}}}}
\def\Sco{{{\operatorname{Sco}}}}
\def\Spm{{\operatorname{Sp}}_{\operatorname{m}}^\times}
\def\U{{{\mathcal{U}}}}
%%%%%%%%%%%%%%%%%%%%%%%%%%
\def\OAf{{{\mathcal{O}}_{A_f}}}
\def\OTAf{{{\mathcal{O}}_{(A_f)^t}}}

%%%%%%%%%%%%
%%%%%%%%%%%%%%%%%%%%%%%%%
\section{Introduction}
%%%%%%%%%%%%%%%%%%%%%%
Flow equivalence relation in two-sided topological Markov shifts 
is one of most interesting and important equivalence relations
in symbolic dynamics as seen in lots of papers \cite{BF}, \cite{BH}, \cite{Franks}, \cite{PS}, etc.
Let $(\bar{X}_A,\bar{\sigma}_A)$ be the two-sided topological Markov shift
defined by an $N\times N$ irreducible matrix $A =[A(i,j)]_{i,j=1}^N$
with entries in $\{0,1\}.$ 
The shift space $\bar{X}_A$ consists of bi-infinite sequences 
$(x_n)_{n\in \Z} \in \{1,\dots,N\}^{\Z}$ of $\{1,\dots,N\}$ 
such that $A(x_n, x_{n+1}) =1$ for all $n \in \Z$. Take and fix a real number 
$\lambda_\circ$ such as 
$0<\lambda_\circ <1$.
The space is a compact metric space  by the metric defined by for 
$x =(x_n)_{n\in \Z} , y=(y_n)_{n\in \Z}$ with $x \ne y$
\begin{equation*}
d(x,y ) 
=  
\begin{cases}
1 & \text{ if  }  x_0 \ne y_0,\\
(\lambda_\circ)^{m} & \text{ if } m=\Max\{ n \mid x_k = y_k \text{ for all }k \text{ with }
|k| < n\}.
\end{cases}
\end{equation*}
The homeomorphism of  shift transformation 
$\bar{\sigma}_A$ on $\bar{X}_A$ is defined by
$\bar{\sigma}_A((x_n)_{n\in \Z}) = (x_{n+1})_{n\in \Z}. $  
Two topological Markov shifts  
$(\bar{X}_A,\bar{\sigma}_A)$
and
$(\bar{X}_B,\bar{\sigma}_B)$
are said to be flow equivalent if they are realized as cross sections with 
their first return maps of a common flow space.
 Parry--Sullivan in \cite{PS} proved that 
$(\bar{X}_A,\bar{\sigma}_A)$
and
$(\bar{X}_B,\bar{\sigma}_B)$
are  flow equivalent if and only if they are realized as discrete cross sections with 
their first return maps of a common topological Markov shift.
Cuntz--Krieger have first found that there is an interesting relation between 
flow equivalence of topological Markov shifts and 
certain purely infinite simple $C^*$-algebras called Cuntz--Krieger algebras  that they introduced in \cite{CK}.
For an irreducible matrix $A$ with entries in $\{0,1\}$, 
let $\OA$ be the Cuntz--Krieger algebra and $\DA$ its canonical maximal abelian $C^*$-subalgebra of $\OA$.
We denote by $\K$ and $\C$ the $C^*$-algebra of compact operators on the separable infinite dimensional Hilbert space $\ell^2(\N)$ and its commutative $C^*$-subalgebra of diagonal operators on $\ell^2(\N)$, respectively.
Cuntz--Krieger proved that for irreducible non-permutation matrices $A$ and $B$,    
if 
$(\bar{X}_A,\bar{\sigma}_A)$
and
$(\bar{X}_B,\bar{\sigma}_B)$
are  flow equivalent, then there exists an isomorphism
$\Phi:\OA\otimes\K \longrightarrow \OB\otimes\K$
of $C^*$-algebras 
such that $\Phi(\DA\otimes\C) = \DB\otimes\C.$ 
Its converse implication holds by \cite{MMKyoto}
(for more general matrices a similar assertion is shown in \cite{CEOR}).
They also proved in \cite{CK} that the first $K$-homology group
$K^1(\OA),$ which is isomorphic to the $K$-group $K_0(\OA)$
as groups,
appears as the Bown--Franks group $\BF(A)$ defined by Bowen--Franks in \cite{BF}, that is an invariant of flow equivalence 
of 
$(\bar{X}_A,\bar{\sigma}_A)$
(\cite{BF}).

There are another kinds of construction of $C^*$-algebras from two-sided topological Markov shifts by using groupoids regarding the Markov shifts as Smale spaces
(\cite{Bowen2},  \cite{Ruelle1}, \cite{Smale}, etc. ).
The construction has been  initiated by D. Ruelle \cite{Ruelle1}, \cite{Ruelle2} and I. Putnam \cite{Putnam1}, \cite{Putnam2}.
 They constructed several kinds of groupoids from each Smale space.
Each of the groupoids yields a $C^*$-algebra.
In this paper, we focus on asymptotic groupoids $G_A^a$ among several groupoids in
studying in  
\cite{KamPutSpiel}, \cite{Putnam1}, \cite{Putnam2}, \cite{PutSp}, etc.  and their semi-direct products
defined below.  
The asymptotic \'etale groupoid $G_A^a$ for $(\bar{X}_A,\bar{\sigma}_A)$
is defined by 
\begin{equation*}
G_A^a := \{ (x, y) \in \bar{X}_A\times \bar{X}_A \mid
   \lim_{n\to{\infty}}d(\sigma_A^n(x), \sigma_A^n(y)) 
 =\lim_{n\to{-\infty}}d(\sigma_A^n(x), \sigma_A^n(y))=0\}
\end{equation*}
with natural groupoid operations and topology (see \cite{Putnam1}). 
It has been shown in \cite{Putnam2} that the groupoid $G_A^a$ is amenable and its $C^*$-algebra
$C^*(G_A^a)$ is stably isomorphic to the tensor product 
$\FTA\otimes\FA$ of the canonical AF-subalgebras 
$\FTA$ and $\FA$  inside the Cuntz--Krieger algebras 
$\OTA$ and $\OA$, respectively.
The semi-direct product $G_A^a\rtimes\Z$ is defined by 
\begin{equation*}
G_A^a\rtimes\Z := 
\{ (x, k-l, y) \in \bar{X}_A\times\Z\times  \bar{X}_A \mid
(\bar{\sigma}_A^k(x),\bar{\sigma}_A^l(x)) \in G_A^a \}
\end{equation*}
with natural groupoid operations and topology (see \cite{Putnam1}). 
It is  \'etale and amenable. 
The groupoid $C^*$-algebra $C^*(G_A^a\rtimes\Z )$
 is called the Ruelle algebra for the Markov shift $(\bar{X}_A,\bar{\sigma}_A)$
and written $\RA$.
Since the unit space 
$(G_A^a\rtimes\Z)^\circ$ is 
$ \{ (x,0,x) \in G_A^a\rtimes\Z \mid x \in \bar{X}_A\}$
that is identified with the shift space $\bar{X}_A$,
the algebra $\RA$ has the commutative $C^*$-algebra $C(\bar{X}_A)$
of continuous functions on $\bar{X}_A$  as a maximal commutative $C^*$-subalgebra. 
It is the crossed product $C^*(G_A^a)\rtimes\Z $
of $C^*(G_A^a)$
induced by the  automorphism of the shift $\bar{\sigma}_A$,
and hence has the dual action written $\rho^A_t, t\in \T$.  
See \cite{Renault1} for general construction of $C^*$-algebras from groupoids.

Following \cite{Putnam1}, let us consider the groupoids $G_A^s$ and $G_A^u$ 
 defined by stable equivalence relation and unstable equivalence relation 
on $(\bar{X}_A,\bar{\sigma}_A)$, respectively,  which are defined by
\begin{align*}
G_A^s = & \{ (x, y) \in \bar{X}_A\times \bar{X}_A \mid
\lim_{n\to{\infty}}d(\bar{\sigma}_A^n(x), \bar{\sigma}_A^n(y) =0\}, \\
G_A^u = & \{ (x,y) \in \bar{X}_A\times \bar{X}_A \mid
\lim_{n\to{-\infty}}d(\bar{\sigma}_A^n(x), \bar{\sigma}_A^n(y) =0\}.
\end{align*}
In \cite{MaPre2017a},\cite{MaPre2017b},
the author introduced the following groupoid
written $G_A^{s,u}\rtimes\Z^2$:
\begin{equation*}
G_A^{s,u}\rtimes\Z^2:= 
\{(x,p,q,y) \in \bar{X}_A\times\Z\times  \Z\times  \bar{X}_A \mid
(\bar{\sigma}_A^p(x), y) \in G_A^s,
(\bar{\sigma}_A^q(x)),y) \in G_A^u \}
\end{equation*}
which has a natural groupoid operations and topology making it \'etale and amenable.
 The groupoid $C^*$-algebra 
$C^*(G_A^{s,u}\rtimes\Z^2)$ is called the extended Ruelle algebra  written $\WRA.$
Since the unit space 
$(G_A^{s,u}\rtimes\Z^2)^\circ$
is $\{(x,0,0, x) \in G_A^{s,u}\rtimes\Z^2 \mid x \in \bar{X}_A\}$
that is identified with the shift space $\bar{X}_A$, 
 the algebra $\WRA$ has $C(\bar{X}_A)$ as a maximal abelian $C^*$-subalgebra.
As in \cite{MaPre2017a},\cite{MaPre2017b}, 
there exists a projection
$E_A$ in the tensor product $\OTA\otimes\OA$ such that 
$E_A(\OTA\otimes\OA)E_A$ is naturally isomorphic to the algebra
$\WRA$, so that the $C^*$-algebra $\WRA$ is regarded as a version of bilateral 
Cuntz--Krieger algebra. 
Let $\alpha_A$ denote the gauge action on the Cuntz--Krieger algebra
$\OA$. 
Under the identification between 
$E_A(\OTA\otimes\OA)E_A$  and $\WRA$,
the tensor product 
$\alpha^{A^t}_r\otimes\alpha^A_s$ for $(r,s) \in \T^2$
yields an action
of $\T^2$ written $\gamma^A_{(r,s)}, (r,s) \in \T^2.$
In \cite[Theorem 1.1]{MaPre2017b},
it was shown that 
the triplet $(\WRA, C(\bar{X}_A), \gamma^A)$ is a complete invariant for the topological conjugacy class of $(\bar{X}_A,\bar{\sigma}_A)$.
For a continuous function $f:\bar{X}_A\longrightarrow \N,$ 
we may define an action $\gamma^{A,f}$  weighted by $f$
on $\WRA.$  
In this paper, we will characterize the flow equivalence class of
$(\bar{X}_A,\bar{\sigma}_A)$ in terms of the stabilized version of
$\WRA$ with the weighted action $\gamma^{A,f}.$
The continuous function $f:\bar{X}_A\longrightarrow \N$ 
exactly corresponds to a ceiling function of discrete suspension.
The main result of this paper is the following theorem.
\begin{theorem}[{Theorem \ref{thm:6.7}}] \label{thm:thm1.1}
Let $A, B$ be irreducible, non-permutation matrices with entries in $\{0,1\}.$
The two-sided topological Markov shifts 
$(\bar{X}_B,\bar{\sigma}_B)$ and 
$(\bar{X}_A,\bar{\sigma}_A)$ are flow equivalent
if and only if 
there exist an irreducible non-permutation matrix $C$ with entries in $\{0,1\}$
and continuous functions
$f_A, f_B:\bar{X}_C \longrightarrow\N$
with its values in positive integers such that 
there exist isomorphisms
$\Phi_A: \WRA\otimes\K \longrightarrow \WRC\otimes\K$
and
$\Phi_B: \WRB\otimes\K \longrightarrow \WRC\otimes\K$
of $C^*$-algebras
such that 
\begin{equation*}
\Phi_A\circ (\gamma^A_{(r,s)}\otimes\id)
=( \gamma^{C,f_A}_{(r,s)}\otimes\id)\circ\Phi_A, \qquad
\Phi_B\circ (\gamma^B_{(r,s)}\otimes\id)
=( \gamma^{C,f_B}_{(r,s)}\otimes\id)\circ\Phi_B
\end{equation*}
for $(r,s) \in \T^2.$
%\begin{gather*}
%\Phi_A(C(\bar{X}_A)\otimes\C) = C(\bar{X}_C)\otimes\C, \qquad
%\Phi_A\circ (\gamma^A_{(r,s)}\otimes\id)
%=( \gamma^{C,f_A}_{(r,s)}\otimes\id)\circ\Phi_A, \\
%\Phi_B(C(\bar{X}_B)\otimes\C) = C(\bar{X}_C)\otimes\C, \qquad
%\Phi_B\circ (\gamma^B_{(r,s)}\otimes\id)
%=( \gamma^{C,f_B}_{(r,s)}\otimes\id)\circ\Phi_B,
%\end{gather*}for $(r,s) \in \T^2.$
\end{theorem}
The above statement exactly corresponds to the situation that  the topological Markov shift 
$(\bar{X}_A,\bar{\sigma}_A)$ is realized as a discrete suspension of 
$(\bar{X}_C,\bar{\sigma}_C)$ by ceiling function $f_A$,
and 
$(\bar{X}_B,\bar{\sigma}_B)$ is realized as a discrete suspension of 
$(\bar{X}_C,\bar{\sigma}_C)$ by ceiling function $f_B$.

As a corollary we have the following.
\begin{corollary}[{Corollary \ref{cor:4.8}}] \label{cor:1.2}
Let $A, B$ be irreducible, non-permutation matrices with entries in $\{0,1\}.$
The two-sided topological Markov shifts 
$(\bar{X}_B,\bar{\sigma}_B)$ and 
$(\bar{X}_A,\bar{\sigma}_A)$ are flow equivalent
if and only if 
there exist 
continuous functions
$f_A:\bar{X}_A \longrightarrow \N$
and
$f_B:\bar{X}_B \longrightarrow \N$
with its values in positive integers, and 
 an isomorphism
$\Phi: \WRA\otimes\K \longrightarrow \WRB\otimes\K$
of $C^*$-algebras such that 
$$
%\Phi( C(\bar{X}_A)\otimes\C) = C(\bar{X}_B)\otimes\C, \qquad
\Phi\circ(\gamma^{A,f_A}_{(r,s)}\otimes\id)
=(\gamma^{B,f_B}_{(r,s)}\otimes\id)\circ\Phi,
\qquad (r,s) \in \T^2.
$$
\end{corollary}

The organization of the paper is the following.

In Section 2, we will briefly recall basic notation and terminology on 
groupoid $C^*$-algebras, Cuntz-Krieger  algebras,
Ruelle algebras and flow equivalence of topological Markov shifts.

In Section 3, a bilateral version of the Krieger's dimension group for topological Markov shifts 
will be studied and called the dimension quadruplet that will be shown to be invariant for shift equivalence of the underlying matrices.

In Section 4, the dimension quadruplet is described by the K-group
of the AF-algebra $C^*(G_A^a)$ of the groupoid $G_A^a.$
As a result, an sufficient condition under  which the two-sided Markov shifts 
$(\bar{X}_B,\bar{\sigma}_B)$ and 
$(\bar{X}_A,\bar{\sigma}_A)$ are flow equivalent
is given in terms of the stabilized action $\gamma^A\otimes \id$ of $\T^2$ on $\WRA\otimes\K$ (Proposition \ref{prop:4.5}).
 
In Section 5, the action $\gamma^{A,f}$ with potential function $f$ on
the algebra $\WRA$ is introduced. 

In Section 6,
we characterize  flow equivalence of  two-sided topological Markov shifts 
in terms of the actions with potential functions  of two-dimensional torus on the extended Ruelle algebras
 $\WRA$.

In Section 7,
we study a detailed statement of
Theorem \ref{thm:thm1.1} (Proposition \ref{prop:7.1}).

Throughout the paper, we denote by $\Zp, \N$
the set of nonnegative integers, the set of positive integers, respectively.

%%%%%%%%%%%%%%%%%%%%%%%%%%%%%%%%%%%%
%%%%%%%%%%%%%%%%%%%%%%%%%%%%%%%%
\section{Preliminaries}
%%%%%%%%%%%%%%%%%%%%%%%%%%%%%%%%%%
In this section, we briefly recall basic notation and terminology on 
the $C^*$-algebras of \'etale groupoids,   Cuntz-Krieger  algebras,
Ruelle algebras and flow equivalence of topological Markov shifts.
In what follows, a square matrix $A =[A(i,j)]_{i,j=1}^N$ is
assumed to be  an  $N\times N$
irreducible, non-permutation matrix with entries in $\{0,1\}.$  

\medskip

{\bf 1. $C^*$-algebras of \'etale groupoids}

Let us construct $C^*$-algebras from \'etale groupoids.
The general theory of the construction of groupoid $C^*$-algebras
was initiated and studied by Renault  \cite{Renault1} (see also \cite{Renault2}, \cite{Renault3}).
 The construction will be used in the following sections. 
Let $G$ be an \'etale groupoid with its unit space $G^\circ$ and range map, source map
$r, s: G \longrightarrow G^\circ$
and $C_c(G)$ denote the $*$-algebra of continuous functions on $G$ with compact support having its product and $*$-involution defined by  
\begin{equation*}
(f*g)(\gamma) = \sum_{\eta; r(\gamma) = r(\eta)} f(\eta)g(\eta^{-1}\gamma), \qquad
f^*(\gamma) =\overline{f(\gamma^{-1})},\qquad f,g \in C_c(G), \,\, \gamma \in G.
\end{equation*}
We denote by $C_0(G^\circ)$ the commutative $C^*$-algebra of continuous functions on $G^\circ$
vanishing at infinity.
The algebra  
$C_c(G)$ has a structure of right $C_0(G^\circ)$-module with 
$C_0(G^\circ)$-valued right inner product given by
\begin{equation*}
(\xi g)(\gamma) = \xi(\gamma) g(r(\gamma)), \qquad
< \xi, \zeta>(t) = \sum_{\eta; t = r(\eta)} \overline{\xi(\eta)}\zeta(\eta)
\end{equation*}
for $\xi, \zeta \in  C_c(G), \, \, g \in C_c(G^\circ), \, \gamma\in G, \, t \in G^\circ.$
The completion of 
$C_c(G)$ by the norm defined by the above inner product is denoted by 
$\ell^2(G),$ which is a Hilbert $C^*$-right module over $C_0(G^\circ).$
The algebra $C_c(G)$ is represented on $\ell^2(G)$ as bounded adjointable 
$C_0(G^\circ)$-right module maps by
$\pi(f)\xi = f*\xi$ for $f \in C_c(G), \xi \in \ell^2(G).$
The closure of $\pi(C_c(G))$ by the operator norm on $\ell^2(G)$
is denoted by $C^*_r(G)$ and called the (reduced) groupoid $C^*$-algebra for the \'etale groupoid $G.$ 
The completeion of $C_c(G)$ by the universal $C^*$-norm is called the (full) groupoid $C^*$-algebra for  $G.$ 
Now we treat the three kinds of groupoids
$G_A^a, G_A^a\rtimes\Z, G_A^{s,u}\rtimes\Z^2$.
They are all \'etale and amenable, so that the two groupoid $C^*$-algebras
 $C^*_r(G)$ and  $C^*(G)$ are canonically isomorphic for such groupoids.
 We do not distinguish them, and write them as 
$C^*(G)$ for $G=G_A^a, G_A^a\rtimes\Z, G_A^{s,u}\rtimes\Z^2$. 
 
\medskip

%%%%%%%%%%%%%%%%%%%%%%%%%%%%%%%%%%%%%%%%%%%%%%%%%
{\bf 2. Cuntz-Krieger  algebras,
Ruelle algebras and extended Ruelle algebras}
%%%%%%%%%%%%%%%%%%%%%%%%%%%%%%%%%%%%%%%%%%%%%%%%%%%%%

The Cuntz--Krieger algebra $\OA$ introduced by Cuntz--Krieger \cite{CK}
is a universal unique $C^*$-algebra 
generated by partial isometries $S_1,\dots,S_N$ subject to the relations:
\begin{equation}
\sum_{j=1}^N S_j S_j^* =1, \qquad S_i^* S_i =\sum_{j=1}^N A(i,j)S_j S_j^*,
\qquad i=1,\dots,N.
\label{eq:CK}
\end{equation}
By the universality for the  relations \eqref{eq:CK} of operators,
the correspondence 
$S_i \longrightarrow \exp(2\pi\sqrt{-1}t)S_i,$
$ i=1,\dots,N$ for each 
$t \in \mathbb{R}/\Z =\T$ yields an automorphism written $\alpha^A_t$ on the $C^*$-algebra $\OA.$ 
The automorphisms $\alpha^A_t, t \in \T$ define an action of $\T$ on $\OA$ called the gauge action.
It is well-known that the fixed point algebra 
$(\OA)^{\alpha^A}$ of $\OA$ under the gauge action
is an AF-algebra written $\FA$.
Let us denote by $B_m(\bar{X}_A)$ the set of admissible words in $\bar{X}_A$ of length 
$m$ and by 
$B_*(\bar{X}_A)$ the set of all admissible words of $\bar{X}_A$.
For $\mu = (\mu_1,\dots,\mu_m) \in B_m(\bar{X}_A)$, 
we write
$S_\mu = S_{\mu_1}\cdots S_{\mu_m}$.
We denote by $\DA$ the $C^*$-subalgebra of $\FA$ generated by projections
$S_\mu S_\mu^*, \mu \in B_*(\bar{X}_A).$

As in \cite{CK} and \cite{CuntzInvent80} (cf. \cite{MaMathScand1998},
\cite{Rosenberg}),
the crossed product $\OA\rtimes_{\alpha^A}\T$ is stably isomorphic to
the AF-algebra $\FA$. Hence the dual action $\hat{\alpha}^A$ on
$\OA\rtimes_{\alpha^A}\T$ induces an automorphism
on $K_0(\FA)$, that is written $\delta_A$.
The triplet $(K_0(\FA),K_0^+(\FA),\delta_A)$
appears as  the (future) dimension triplet written
$(\Delta_A,\Delta_A^+,\delta_A)$ for $A$
defined by W. Krieger \cite{Krieger1980}. 
For the transposed  matrix $A^t$ of $A$, 
we similarly consider the  Cuntz--Krieger algebra
$\OTA$ and its AF-subalgebra $\FTA$.
Let us denote by 
 $T_1,\dots,T_N$ the generating partial isometries of $\OTA$ which satisfy 
the relations:
\begin{equation}
\sum_{i=1}^N T_i T_i^* =1, \qquad T_j^* T_j =\sum_{i=1}^N A(i,j)T_i T_i^*,\quad j=1,\dots,N.
\label{eq:CKt}
\end{equation}
For $\xi =(\xi_1,\dots,\xi_k) \in B_k(\bar{X}_A),$
we denote by $\bar{\xi}$ the transposed word
$(\xi_k,\dots,\xi_1) $ which belongs to $B_k(\bar{X}_{A^t}),$
and write
$T_{\bar{\xi}} = T_{\xi_k}\cdots T_{\xi_1}$.

Define the projection 
$E_A \in \FTA\otimes\FA$ by setting
\begin{equation*}
E_A = \sum_{j=1}^N T_j^* T_j \otimes S_j S_j^*
\end{equation*}
which coincides with
$ \sum_{i=1}^N T_i T_i^* \otimes S_i^* S_i$
because of the equalities \eqref{eq:CK} and \eqref{eq:CKt}.
Let $G_A^a, G_A^a\rtimes\Z, G_A^{s,u}\rtimes\Z^2 $ 
denote the \'etale amenable groupoids 
stated in Section 1.
As in \cite[Proposition 2.1]{MaPre2017b},
we know the following lemma.
\begin{lemma}\hspace{7cm}
\begin{enumerate}
\renewcommand{\theenumi}{\roman{enumi}}
\renewcommand{\labelenumi}{\textup{(\theenumi)}}
\item
The groupoid $C^*$-algebra 
$C^*(G_A^a)$ is canonically isomorphic to the $C^*$-subalgebra of
 $\FTA\otimes\FA$
generated by elements $T_{\bar{\xi}}T_{\bar{\eta}}^* \otimes S_\mu S_\nu^*$ 
where
$\mu =(\mu_1,\dots,\mu_m), \,\nu =(\nu_1,\dots,\nu_n) \in B_*(\bar{X}_A), \, $
$\bar{\xi} =(\xi_k,\dots,\xi_1), \, \bar{\eta} =(\eta_l,\dots,\eta_1) 
\in B_*(\bar{X}_{A^t})$
satisfying 
$A(\xi_k, \mu_1) = A(\eta_l,\nu_1) = 1$
and
$k=l,\, m=n.$
%\begin{align*}
%&\mu =(\mu_1,\dots,\mu_m), \,\nu =(\nu_1,\dots,\nu_n) \in B_*(\bar{X}_A), \\
%&\bar{\xi} =(\xi_k,\dots,\xi_1), \, \bar{\eta} =(\eta_l,\dots,\eta_1) 
%\in B_*(\bar{X}_{A^t}), \\ 
%&A(\xi_k, \mu_1) = A(\eta_l,\nu_1) = 1, \,\, k=l,\, m=n.
%\end{align*}
Hence 
$C^*(G_A^a)$ is canonically isomorphic to the $C^*$-algebra 
 $E_A(\FTA\otimes\FA)E_A.$
\item
The Ruelle algebra $\RA=C^*(G_A^a\rtimes\Z)$
 is canonically isomorphic to the $C^*$-subalgebra of
 $\OTA\otimes\OA$
generated by elements $T_{\bar{\xi}}T_{\bar{\eta}}^* \otimes S_\mu S_\nu^*$ 
where
$\mu =(\mu_1,\dots,\mu_m), \,\nu =(\nu_1,\dots,\nu_n) \in B_*(\bar{X}_A), $
$\bar{\xi} =(\xi_k,\dots,\xi_1), \, \bar{\eta} =(\eta_l,\dots,\eta_1) \in B_*(\bar{X}_{A^t})$
satisfying
$A(\xi_k, \mu_1) = A(\eta_l,\nu_1) = 1$
and
$m+k=n+l.
$
%\begin{align*}
%&\mu =(\mu_1,\dots,\mu_m), \,\nu =(\nu_1,\dots,\nu_n) \in B_*(\bar{X}_A), \\
%&\bar{\xi} =(\xi_k,\dots,\xi_1), \, \bar{\eta} =(\eta_l,\dots,\eta_1) 
%\in B_*(\bar{X}_{A^t}), \\ 
%&A(\xi_k, \mu_1) = A(\eta_l,\nu_1) = 1, \,\, m+k=n+l.
%\end{align*}
\item
The extended Ruelle algebra $\WRA=C^*(G_A^{s,u}\rtimes\Z^2)$
 is canonically isomorphic to the $C^*$-subalgebra of
 $\OTA\otimes\OA$
generated by elements $T_{\bar{\xi}}T_{\bar{\eta}}^* \otimes S_\mu S_\nu^*$ 
where
$\mu =(\mu_1,\dots,\mu_m), \,\nu =(\nu_1,\dots,\nu_n) \in B_*(\bar{X}_A),$
$\bar{\xi} =(\xi_k,\dots,\xi_1), \, \bar{\eta} =(\eta_l,\dots,\eta_1) \in B_*(\bar{X}_{A^t})$
satisfying
$A(\xi_k, \mu_1) = A(\eta_l,\nu_1) = 1.
$
%\begin{align*}
%&\mu =(\mu_1,\dots,\mu_m), \,\nu =(\nu_1,\dots,\nu_n) \in B_*(\bar{X}_A), \\
%&\bar{\xi} =(\xi_k,\dots,\xi_1), \, \bar{\eta} =(\eta_l,\dots,\eta_1) 
%\in B_*(\bar{X}_{A^t}), \\ 
%&A(\xi_k, \mu_1) = A(\eta_l,\nu_1) = 1
%\end{align*}
Hence 
$\WRA$ is canonically isomorphic to the $C^*$-algebra 
 $E_A(\OTA\otimes\OA)E_A.$
\end{enumerate}
\end{lemma}
  Under the identification between 
$E_A(\OTA\otimes\OA)E_A$  and $\WRA$,
the tensor product 
$\alpha^{A^t}_r\otimes\alpha^A_s$ of gauge actions on $\OTA$ and $\OA$
yields an action
of $\T^2$ on $\WRA$ written $\gamma^A_{(r,s)}, (r,s) \in \T^2,$
because
$\gamma^A_{(r,s)}(E_A) = E_A.$
We write
$$
\delta^A_t:= \alpha^{A^t}_t\otimes\alpha^A_t,\qquad
\rho^A_t := \gamma^A_{(-\frac{t}{2},\frac{t}{2})}, \qquad
t \in \T.
$$
\begin{lemma}\hspace{6cm}
\begin{enumerate}
\renewcommand{\theenumi}{\roman{enumi}}
\renewcommand{\labelenumi}{\textup{(\theenumi)}}
\item
The restriction of the action $\rho^A_t, t \in \T$ to the subalegera 
$\RA$ 
is regarded as the dual action on $\RA$ under a natural identification 
between $\RA$ and the crossed product $C^*(G_A^a)\rtimes\Z$.
Hence the fixed point algebra 
$(\RA)^{\rho^A}$ is isomorphic to $C^*(G_A^a).$
 \item
The fixed point algebra $(\WRA)^{\delta^A}$ of $\WRA$ under $\delta^A$ is 
isomorphic to $\RA$, so that 
 the fixed point algebra
$(\WRA)^{\gamma^A}$ of $\WRA$ under $\gamma^A$ is 
isomorphic to $C^*(G_A^a).$ 
\end{enumerate}
\end{lemma}

\medskip

%%%%%%%%%%%%%%%%%%%%%%%%%%%%%%%%%%%%%%%%%%
{\bf 3. Suspension and flow equivalence} 
%%%%%%%%%%%%%%%%%%%%%%%%%%%%%%%%%%%%%%%%%

We will briefly review discrete suspension of topological Markov shifts.
Let $f:\bar{X}_A\longrightarrow \N$ 
be a continuous function on the shift space $\bar{X}_A$
with its values in positive integers.
Let $f(\bar{X}_A) = \{1,2,\dots,L\}.$
Put
$X_j = \{x \in \bar{X}_A\mid f(x) =j\}, j=1,\dots,L.$
Define the suspension space
$\bar{X}_{A, f} = \cup_{j=1}^L X_j \times \{0,1,\dots, j-1\}$
with transformation 
$\bar{\sigma}_{A,f}$ on $\bar{X}_{A, f} $
by
\begin{equation*}
\bar{\sigma}_{A,f} ([x, k] )
= \begin{cases}
[x,k+1] & \text{ if } 0\le k \le j-2, \\
[\bar{\sigma}_A(x), 0] & \text{ if } k=j-1
\end{cases}
\end{equation*}
for $[x,k]\in X_j\times\{0,1,\dots,j-1\}.$
The resulting topological dynamical system
$(\bar{X}_{A, f}, \bar{\sigma}_{A,f}) $
is called the discrete suspension of $(\bar{X}_A, \bar{\sigma}_A)$
by ceiling function $f$, which is homeomorphic to a topological Markov shift.
If in particular 
the function $f: \bar{X}_A \longrightarrow \N$
depends only on the $0$th coordinate of $\bar{X}_A$,
then $f$ is written  
$f = \sum_{j=1}^N f_j \chi_{U_j(0)}$ for some integers $f_j \in \Z,$ 
where  $\chi_{U_j(0)}$ is the characteristic function of the cylinder set
$$  
U_j(0)  =\{ (x_n)_{n \in \Z} \in \bar{X}_A \mid x_0 = j \},
\qquad j=1,\dots,N.
$$
Put $m_j  =f_j -1$ for $j=1,\dots,N.$
Let
$\G =(\V,\E)$ be the directed graph defined by the matrix $A$
with the vertex set $\V = \{1,2,\dots,N\}.$
An edge of $\G$ is defined by a pair 
$(i,j)$ of vertices $i,j=1,\dots,N$ such that 
$A(i,j) =1,$ 
whose source is $i$ and the terminal is $j.$
The set of such pairs $(i,j)$ 
is the edge set $\E.$
Consider a new graph 
$\G_f =(\V_f,\E_f)$ with its transition matrix $A_f$  
from the graph $\G =(\V,\E)$ and the function $f$
such that 
$\V_f =\cup_{j=1}^N \{ j_0, j_1, j_2,\dots, j_{m_j} \}$
and
if $A(j,k) =1,$ then
\begin{equation}
A_f(j_0,j_1) =A_f(j_1,j_2) =\cdots =A_f(j_{mj-1},j_{m_j}) =A_f(j_{m_j}, k_0) =1.
\label{eq:2.3}
\end{equation}
For other pairs $(j_i, j'_{i'}) \in \V_f\times\V_f,$
we define $A_f(j_i, j'_{i'}) =0.$ 
Hence the size of the matrix $A_f$ is
$(f_1+f_2+ \cdots +f_N) \times (f_1+f_2+ \cdots +f_N). $ 
Then the 
discrete suspension 
$(\bar{X}_{A, f}, \bar{\sigma}_{A,f}) $
is nothing but 
the topological Markov shift
$( \bar{X}_{A_ f}, \bar{\sigma}_{A_f}) $
defined by the matrix 
$A_f$.

Two topological Markov shifts  
%$(\bar{X}_A,\bar{\sigma}_A)$ and $(\bar{X}_B,\bar{\sigma}_B)$
are said to be flow equivalent if they are realized as cross sections with 
their first return maps of a common one-dimensional flow space.
 Parry--Sullivan in \cite{PS} proved that 
$(\bar{X}_A,\bar{\sigma}_A)$
and
$(\bar{X}_B,\bar{\sigma}_B)$
are  flow equivalent if and only if
there exist another topological Markov shift 
$(\bar{X}_C,\bar{\sigma}_C)$ for some matrix $C$
and continuous maps
$f_A, f_B: \bar{X}_C\longrightarrow \N$
such that 
$(\bar{X}_A,\bar{\sigma}_A)$
is topologically conjugate to the discrete suspension 
$(\bar{X}_{C, f_A}, \bar{\sigma}_{C, f_A}) $
and
$(\bar{X}_B,\bar{\sigma}_B)$
is topologically conjugate to the discrete suspension 
$(\bar{X}_{C, f_B}, \bar{\sigma}_{C, f_B}).$

Cuntz--Krieger have first found that there is an interesting relation between 
flow equivalence  of topological Markov shifts and 
 Cuntz--Krieger algebras in \cite{CK}.
We denote by $\K$ and $\C$ the $C^*$-algebra of compact operators on the separable infinite dimensional Hilbert space $\ell^2(\N)$ and its commutative $C^*$-subalgebra of diagonal operators on $\ell^2(\N)$, respectively.
Cuntz--Krieger proved that for irreducible non-permutation matrices $A$ and $B$,    
if 
$(\bar{X}_A,\bar{\sigma}_A)$
and
$(\bar{X}_B,\bar{\sigma}_B)$
are  flow equivalent, then there exists an isomorphism
$\Phi:\OA\otimes\K \longrightarrow \OB\otimes\K$
of $C^*$-algebras 
such that $\Phi(\DA\otimes\C) = \DB\otimes\C.$ 
Its converse implication holds by \cite{MMKyoto}
(for more general matrices a similar assertion is shown in \cite{CEOR}).

In this paper, we will study flow equivalence of topological Markov shifts in terms of 
the extended Ruelle algebras with its action $\gamma^A$ of $\T^2.$

%%%%%%%%%%%%%%%%%%%%%%%%%%%%%%%%%%%%%%%%%%%%
%%%%%%%%%%%%%%%%%%%%%%%%%
\section{Bilateral dimension groups}
%%%%%%%%%%%%%%%%%%%%%%
%%%%%%%%%%%%%%%%%%%%%%%%%
We keep an irreducible, non-permutation matrix 
$A =[A(i,j)]_{i,j=1}^N$ with entries in $\{0,1\}.$
Following W. Krieger \cite{Krieger1980} (cf. \cite{Krieger1977}, \cite{Krieger1979}, \cite{Effros}, etc.),
the dimension group $(\Delta_A, \Delta_A^+)$
 are defined as  an ordered group by the inductive limits
\begin{equation*}
\Delta_A =  \Z^N \overset{A^t}{\longrightarrow}\Z^N \overset{A^t}{\longrightarrow}\cdots, \qquad
\Delta_A^ + =  \Zp^N \overset{A^t}{\longrightarrow}\Zp^N \overset{A^t}{\longrightarrow}\cdots. 
\end{equation*}
The group $\Delta_A $ is identified with the equivalence classes
of $\cup_{n=0}^{\infty}\{(v,n)\mid v \in \Z^N, n \in \Zp\}$
 by the equivalence relation generated by
$(v, n) \sim (A^t v, n+1)$.
The equivalence class of $(v,n)$ is denoted by $[v,n].$
The dimension drop automorphism
$\delta_A$ on $(\Delta_A, \Delta_A^+)$ is defined 
by
$\delta_A([v,n]) = [(v,n+1)]$ for $[v,n] \in \Delta_A.$
The triplet
$(\Delta_A, \Delta_A^+,\delta_A )$
is called the (future) dimension triplet for 
the topological Markov shift $(\bar{X}_A, \bar{\sigma}_A)$.
We similarly have the (future) dimension triplet
$(\Delta_{A^t}, \Delta_{A^t}^+,\delta_{A^t} )$
 for 
the topological Markov shift $(\bar{X}_{A^t}, \bar{\sigma}_{A^t})$ for
the matrix $A^t$, which is called the (past) dimension triplet for  
$(\bar{X}_A, \bar{\sigma}_A).$
Hence we have two dimension triplets
$(\Delta_A, \Delta_A^+,\delta_A )$ and 
$(\Delta_{A^t}, \Delta_{A^t}^+,\delta_{A^t} )$
for the matrix $A$.

Let $e_i \in \Z^N$ be the vector of $\Z^N$ whose $i$th component is $1,$
other components are zeros.
We will define a specific element
$\tilde{u}_A$ in $\Delta_{A^t}\otimes \Delta_{A}$
by setting
\begin{equation*}
\tilde{u}_A := \sum_{i,j=1}^N [e_j,1]\otimes A(j,i)[e_i,1] \in \Delta_{A^t}\otimes \Delta_{A}.
\end{equation*}
We then see that 
\begin{align*}
\tilde{u}_A 
= & \sum_{j=1}^N \left( [e_j,1]\otimes [\sum_{i=1}^NA(j,i)e_i, 1]\right) 
=  \sum_{j=1}^N [e_j,1]\otimes [A^t e_j, 1] 
=  (\id \otimes \delta_A^{-1})\sum_{j=1}^N ([e_j,1]\otimes [e_j, 1]) \\
\intertext{and}
\tilde{u}_A 
= & \sum_{i=1}^N \left( [\sum_{j=1}^NA(j,i)e_j,1]\otimes [e_i, 1]\right) 
=  \sum_{j=1}^N [Ae_i,1]\otimes [e_i, 1] 
=  (\delta_{A^t}^{-1}\otimes\id)\sum_{i=1}^N ([e_i,1]\otimes [e_i, 1] ).
\end{align*}
Define 
an automorphism 
$\tilde{\delta}_A : \Delta_{A^t}\otimes \Delta_{A}\longrightarrow  \Delta_{A^t}\otimes \Delta_{A}$
by $\tilde{\delta}_A =\delta_{A^t}^{-1}\otimes\delta_A$.
It satisfies
$$
\tilde{\delta}_A([u,n]\otimes[v,m]) = [Au,n]\otimes[v,m+1],
\qquad [u,n]\otimes[v,m] \in \Delta_{A^t}\otimes \Delta_A.
$$
We set the abelian group
$\tilde{\Delta}_A=\Delta_{A^t}\otimes \Delta_A
$ 
with its positive cone
$
\tilde{\Delta}_A^+=\Delta_{A^t}^+\otimes \Delta_A^+.
$
\begin{definition}
The quadruplet 
$( \tilde{\Delta}_A, \tilde{\Delta}_A^+, \tilde{\delta}_A, \tilde{u}_A)$
is called the {\it dimension quadruplet}\/ for the two-sided topological Markov shift
$(\bar{X}_A,\bar{\sigma}_A)$.
\end{definition}
We note that a bilateral version of the  dimension groups first appeared in Krieger's paper
\cite{Krieger1977} (cf.  \cite{Krieger1979}, \cite{Krieger1980}).
\begin{lemma}
$\tilde{\delta}_A(\tilde{u}_A) =\tilde{u}_A.$
\end{lemma}
\begin{proof}
Since
\begin{equation*}
\tilde{u}_A  
=  (\id \otimes \delta_A^{-1})\sum_{j=1}^N ([e_j,1]\otimes [e_j, 1]) 
=  (\delta_{A^t}^{-1}\otimes\id)\sum_{i=1}^N ([e_i,1]\otimes [e_i, 1]),
\end{equation*}
and
$\tilde{\delta}_A =\delta_{A^t}^{-1}\otimes\delta_A,$
the assertion is immediate.
\end{proof}
We will next show that the dimension quadruplet 
$( \tilde{\Delta}_A, \tilde{\Delta}_A^+, \tilde{\delta}_A, \tilde{u}_A)$
is invariant under shift equivalence of the underlying matrices $A$.
The notion of shift equivalence in square matrices with entries in nonnegative integers has been introduced by W. F. Williams \cite{Williams}.
Two matrices $A$ and $B$ are said to be shift equivalent if there exist rectangular matrices
$H, K$  with entries in nonnegative integers and  a positive integer
$\ell$ such that   
\begin{equation}
A^\ell = H K, \quad B^\ell = KH, \quad AH = HB,\quad KA = BK.\label{eq:SE}
\end{equation}
W. Krieger has proved in \cite{Krieger1980} that two matrices $A$ and $B$
are shift equivalent if and only if their dimension triplet
$(\Delta_A, \Delta_A^+, \delta_A)$ and
$(\Delta_B, \Delta_B^+, \delta_B)$
are isomorphic.
 The following result has been already proved by 
C. G. Holton \cite[Proposition 6.7]{Holton}
for primitive matrices by using Rohlin property of automorphisms on the AF-algebras
$C^*(G_A^a)$. 
The proof given below does not use $C^*$-algebra theory,
nor does assume for the matrices to be primitive.
\begin{proposition}[{C. G.  Holton \cite[Proposition 6.7]{Holton}}] 
Suppose that $A$ and $B$ are shift equivalent.
Then there exists an isomorphism 
$\Phi: \tilde{\Delta}_A \longrightarrow \tilde{\Delta}_B$ which yields 
an isomorphism between the dimension quadruplets
$( \tilde{\Delta}_A, \tilde{\Delta}_A^+, \tilde{\delta}_A, \tilde{u}_A)$
and
$( \tilde{\Delta}_B, \tilde{\Delta}_B^+, \tilde{\delta}_B, \tilde{u}_B.)$
\end{proposition}
\begin{proof}
Let $A$ and $B$ be $N\times N$ matrix and $M\times M$ matrix, respectively.  
Assume that  there exist rectangular matrices
$H, K$  with entries in nonnegative integers and  a positive integer
$\ell$ satisfying \eqref{eq:SE}.
Define 
\begin{align*}
\Phi_r: \Delta_A \longrightarrow \Delta_B \quad 
&\text{ by } \quad \Phi_r([v,k]) = [H^tv,k],\\
\Phi_l: \Delta_{A^t} \longrightarrow \Delta_{B^t} \quad 
&\text{ by } \quad \Phi_l([v,k]) = [Kv, k+\ell], \\
\intertext{so that}
\Phi_r^{-1}: \Delta_B \longrightarrow \Delta_A \quad 
&\text{ satisfies } \quad \Phi_r^{-1}([u,j]) = [K^tu, j+\ell],\\
\Phi_l^{-1}: \Delta_{B^t} \longrightarrow \Delta_{A^t} \quad 
&\text{ satisfies } \quad \Phi_l^{-1}([u,j]) = [Hu,j]. 
\end{align*}
As in \cite{Krieger1980},
$\Phi_r: \Delta_A \longrightarrow \Delta_B
$
and
$
\Phi_l: \Delta_{A^t} \longrightarrow \Delta_{B^t} 
$
yield isomorphisms for each such that 
\begin{gather*}
\Phi_r(\Delta_A^+) =\Delta_B^+, \qquad
\Phi_r \circ \delta_A = \delta_B\circ\Phi_r, \\
\Phi_l(\Delta_{A^t}^+) =\Delta_{B^t}^+, \qquad
\Phi_l \circ \delta_{A^t} = \delta_{B^t}\circ\Phi_l.
 \end{gather*} 
Hence they induce isomorphisms
\begin{align*}
\Phi_r: (\Delta_A, \Delta_A^+, \delta_A)  
& \longrightarrow  (\Delta_B, \Delta_B^+, \delta_B),\\
\Phi_l: (\Delta_{A^t}, \Delta_{A^t}^+, \delta_{A^t})  
& \longrightarrow  (\Delta_{B^t}, \Delta_{B^t}^+, \delta_{B^t}).
\end{align*}
We define 
$\Phi = \Phi_l \otimes \Phi_r :\tilde{\Delta}_A \longrightarrow \tilde{\Delta}_B.$
Let $f_l\in \Z^M$ be the vector
whose $l$th component is $1$, and other components are zeros.  
It then follows that
\begin{align*}
\Phi(\tilde{u}_A) 
=& \sum_{i,j=1}^N \Phi_l( [e_j,1])\otimes\Phi_r(A(j,i)[e_i,1]) \\
=& \sum_{i,j=1}^N [Ke_j,1+\ell]\otimes A(j,i)[H^t e_i,1] \\
%=& \sum_{j=1}^N [Ke_j,1+\ell]\otimes [\sum_{i=1}^NA(j,i)H^t e_i,1]) \\
=& \sum_{j=1}^N [Ke_j,1+\ell]\otimes [(AH)^t e_j,1] \\
=& \sum_{j=1}^N [Ke_j,1+\ell]\otimes [\sum_{l=1}^M (AH)(j,l)f_l,1] \\
=& \sum_{l=1}^M \sum_{j=1}^N 
[
{\begin{bmatrix}
K(1,j)(AH)(j,l)\\
K(2,j)(AH)(j,l)\\
\vdots\\
K(M,j)(AH)(j,l)\\
\end{bmatrix}},
1+\ell]\otimes [f_l,1] \\
=& \sum_{l=1}^M 
[(KAH)f_l,1+\ell]\otimes [f_l,1] \\
=& \sum_{l=1}^M 
[(BKH)f_l,1+\ell]\otimes [f_l,1] \\
=& \sum_{l=1}^M 
[B^{\ell+1}f_l,1+\ell]\otimes [f_l,1] \\
=& \sum_{l=1}^M 
[Bf_l,1]\otimes [f_l,1]
=\tilde{u}_B.
\end{align*}
 \end{proof}
R. F. Williams characterized topological conjugate 
two-sided topological Markov shifts
$(\bar{X}_A, \bar{\sigma}_A)$
and 
$(\bar{X}_B, \bar{\sigma}_B)$ in terms of an equivalence relation of its underlying matrices,
called strong shift equivalence (\cite{Williams}).
Two square matrices $A$ and $B$ with entries in nonnegative integers are said to be elementary equivalent 
if there exist rectangular matrices $C, D$ with entries in nonnegative integers 
such that $A = CD, B = DC$. 
If two matrices are connected by a finite chain of elementary equivalences, they are said to be strong shift equivalent. 
R. F. Williams proved that      
two-sided topological Markov shift
$(\bar{X}_A, \bar{\sigma}_A)$
and 
$(\bar{X}_B, \bar{\sigma}_B)$
are topologically conjugate if and only if the matrices $A$ and $B$ are strong shift equivalent (\cite{Williams}).
Since strong shift equivalence is weaker than shift equivalence,
by virtue of the Williams' result,  we have 
\begin{proposition}
The dimension quadruplet 
$( \tilde{\Delta}_A, \tilde{\Delta}_A^+, \tilde{\delta}_A, \tilde{u}_A)$
is invariant under topological conjugacy of two-sided topological Markov shift
$(\bar{X}_A,\bar{\sigma}_A).$
\end{proposition}

%%%%%%%%%%%%%%%%%%%%%%%%%%%%%%%%%%%%%%%%%%%%%%%%%%%%
%%%%%%%%%%%%%%%%%%%%%%%%%%%%%%%%%%%%%%%%%%%%%%%%
\section{Dimension quadruplets and AF-algebras}
%%%%%%%%%%%%%%%%%%%%%%%%%%%%%%%%%%%%%%%%%%%%%%%%

In this section, we will study the dimension quadruplet 
$(\tilde{\Delta}_A,\tilde{\Delta}_A^+, \tilde{\delta}_A,\tilde{u}_A)$
by using K-theory for $C^*$-algebras.
D. B. Killough--I. F. Putnam in \cite{KilPut} have deeply studied ring and module structure
of the AF-algebras $C^*(G_A^s)$ as well as $C^*(G_A^a)$
from a different view point from ours below.  
Recall that $\K$ denotes the $C^*$-algebra of compact operators on the separable infinite dimensional Hilbert space $H =\ell^2(\N).$
%%%%%%%%%%%%%%%%%%%%%%%%%%%%%%%%%%%%%%%%%%%%%%%%%%%%%
\begin{lemma}\label{lem:4.1}
Let $A$ be an irreducible, non-permutation matrix with entries in $\{0, 1\}$.
\begin{enumerate}
\renewcommand{\theenumi}{\roman{enumi}}
\renewcommand{\labelenumi}{\textup{(\theenumi)}}
\item
There exists a projection $p_0$ in the crossed product 
 $\WRA\rtimes_{\gamma^A}\T^2$
of $\WRA$ by $\gamma^A$
such that 
$p_0(\WRA\rtimes_{\gamma^A}\T^2)p_0$
is isomorphic $C^*(G_A^a).$ 
Hence
$\WRA\rtimes_{\gamma^A}\T^2$ is stably isomorphic to the AF-algebra
$C^*(G_A^a).$
\item
The inclusion $\iota_A:p_0(\WRA\rtimes_{\gamma^A}\T^2)p_0 
\hookrightarrow \WRA\rtimes_{\gamma^A}\T^2$
induces an isomorphism
\begin{equation*}
\iota_{A*}: K_0(C^*(G_A^a)) \longrightarrow K_0(\WRA\rtimes_{\gamma^A}\T^2)
\end{equation*}
on K-theory where
 $C^*(G_A^a)$ is identified with $p_0(\WRA\rtimes_{\gamma^A}\T^2)p_0.$
\end{enumerate}
\end{lemma}
\begin{proof}
(i)
The fixed point algebra
$(\WRA)^{\gamma^A}$ of 
$\WRA$ under $\gamma^A$ 
coincides with the fixed point algebra
$ (E_A(\OTA\otimes\OA)E_A)^{\alpha^{A^t}\otimes\alpha^A} 
$
which is nothing but 
$ E_A(\FTA\otimes\FA)E_A.$
Hence
$(\WRA)^{\gamma^A}$ 
is identified with
$C^*(G_A^a).$ 
Let $p_0 \in L^1(\T^2, \WRA)$ be the projection defined by 
$p_0(r,s) =1$ for all $(r,s) \in \T^2.$
We know that $p_0$ is a full projection in 
$\WRA\rtimes_{\gamma^A}\T^2$
and
\begin{equation*}
p_0(\WRA\rtimes_{\gamma^A}\T^2)p_0 = (\WRA)^{\gamma^A}  = C^*(G_A^a)
\end{equation*}
by \cite{Rosenberg} or a similar manner to \cite{MaMathScand1998}.
This shows that the algebra  
$\WRA\rtimes_{\gamma^A}\T^2$ is stably isomorphic to the AF-algebra
$C^*(G_A^a)$ by \cite{Brown}.

(ii)
By \cite{Brown},
there exists a partial isometry 
$v_A$ in the multiplier algebra $M(\WRA\rtimes_{\gamma^A}\T^2\otimes\K)$
of $\WRA\rtimes_{\gamma^A}\T^2\otimes\K$
such that $v_A^*v_A = p_0,\, v_Av_A^* =1.$ 
Put
$
\psi_A = \Ad(v_A): p_0(\WRA\rtimes_{\gamma^A}\T^2)p_0\otimes\K
\longrightarrow \WRA\rtimes_{\gamma^A}\T^2\otimes\K, 
$
which is an isomorphism of $C^*$-algebras.
We then have 
for a projection  
$p_0 f p_0 \otimes q \in \WRA\rtimes_{\gamma^A}\T^2\otimes\K,$
\begin{align*}
(\iota_A\otimes\id)_*([p_0 f p_0 \otimes q])
=& [p_0 f p_0 \otimes q] \\
=& [v_A^*v_A(p_0 f p_0 \otimes q)v_A^*v_A] \\
=& [v_A(p_0 f p_0 \otimes q)v_A^*] \\
=& \psi_{A*}([p_0 f p_0 \otimes q]).
\end{align*}
Hence $\iota_{A*} =\psi_{A*}:K_0(C^*(G_A^a)) \longrightarrow
 K_0(\WRA\rtimes_{\gamma^A}\T^2)$
is an isomorphism.  
\end{proof}
Let us denote by
$\hat{\gamma}^A$ the dual action 
of the crossed product
$\WRA\rtimes_{\gamma^A}\T^2.$
Under the identifications
\begin{equation*}
C^*(G_A^a)= (\WRA)^{\gamma^A} =p_0(\WRA\rtimes_{\gamma^A}\T^2)p_0,
\end{equation*}
we define an action $\beta$ of $\Z^2$ on $K_0(C^*(G_A^a))$
by 
\begin{equation*}
\beta_{(m,n)}:=\iota_{A*}^{-1}\circ\hat{\gamma}^A_{(m,n)*}\circ\iota_{A*}:
K_0(C^*(G_A^a)) \longrightarrow K_0(C^*(G_A^a)), \qquad (m,n) \in \Z^2 
\end{equation*}
such that  the diagram
$$
\begin{CD}
K_0(\WRA\rtimes_{\gamma^A}\T^2) @>\hat{\gamma}^A_{(m,n) *} >> K_0(\WRA\rtimes_{\gamma^A}\T^2) \\
@A{\iota_{A*} }AA  @AA{\iota_{A*}}A \\
K_0(C^*(G_A^a)) @> \beta_{(m,n)} >> K_0(C^*(G_A^a)) 
\end{CD}
$$
is commutative.
%%%%%%%%%%%%%%%%%%%%%%%%%%%%%%%%%%%%%%%%%
%%%%%%%%%%%%%%%%%%%%%%%%%%%%%%%%

Let $U_A =\sum_{i=1}^NT_i^*\otimes S_i$ in $\OTA\otimes\OA.$
As in \cite{MaPre2017a},
 $U_A$ is a unitary in $\RA$ and hence in $\WRA$, so that 
$U_A U_A^* = U_A^* U_A = E_A.$ 
We denote by $1_{C^*(G_A^a)}$
the unit of the $C^*$-algebra $C^*(G_A^a).$
By \cite{CuntzInvent80} and \cite{CK}
(see also \cite{Effros}, \cite{Krieger1979}, \cite{Krieger1980}),
the ordered group $\Delta_A$ is naturally identified with the K-group
$K_0(\FA)$.
\begin{lemma}
There exists an isomorphism
$\varphi_A: C^*(G_A^a)\otimes\K \longrightarrow \FTA\otimes\FA\otimes\K$
of $C^*$-algebras such that the induced isomorphism
\begin{equation*}
\varphi_{A*}:
K_0(C^*(G_A^a)) \longrightarrow K_0(\FTA\otimes\FA) =K_0(\FTA)\otimes K_0(\FA)
\end{equation*}
satisfies
\begin{gather*}
\varphi_{A*}([1_{C^*(G_A^a)}])= [E_A], \qquad
\varphi_{A*} \circ \Ad(U_A)_*  = \tilde{\delta}_A\circ\varphi_{A*}, \\
\varphi_{A*} \circ \beta_{(m,n)}\circ \varphi_{A*}^{-1} 
 = \delta_{A^t}^m \otimes\delta_A^n, \quad (m,n)\in \Z^2.
\end{gather*}
Hence the diagrams
$$
\begin{CD}
K_0(\WRA\rtimes_{\gamma^A}\T^2) @>\hat{\gamma}^A_{(m,n) *} >> K_0(\WRA\rtimes_{\gamma^A}\T^2) \\
@A{\iota_{A*} }AA  @AA{\iota_{A*}}A \\
K_0(C^*(G_A^a)) @> \beta_{(m,n)} >> K_0(C^*(G_A^a)) \\
@V{\varphi_{A*} }VV  @VV{\varphi_{A*}}V \\
K_0(\FTA)\otimes K_0(\FA)
@> \delta_{A^t}^m \otimes\delta_A^n >> 
K_0(\FTA)\otimes K_0(\FA)
\end{CD}
$$
are commutative.
\end{lemma}
\begin{proof}
Since the $C^*$-algebras 
$\FTA\otimes\FA$ is simple, 
the projection $E_A$ is full in $\FTA\otimes\FA$
By using Brown's theorem \cite{Brown},
there exists an isometry 
$u_A$ in the multiplier algebra
$M((\FTA\otimes\FA)\otimes\K)$
of
$(\FTA\otimes\FA)\otimes\K$
such that 
$u_A^* u_A =1, u_A u_A^* = E_A\otimes1_H.$
Define an isomorphism
$$
\varphi_A =\Ad(u_A^*):   
C^*(G_A^a)\otimes\K(=E_A(\FTA\otimes\FA)E_A\otimes\K)
 \longrightarrow \FTA\otimes\FA\otimes\K.
$$
Let  $p_1$ be a rank one projection in $\K$.
We then have
\begin{align*}
\varphi_{A*}([1_{C^*(G_A^a)}] )
=& \varphi_{A*}([E_A\otimes p_1]) \\
=& [u_A^*(E_A\otimes p_1)u_A] \\
=& [((E_A\otimes p_1)u_A)((E_A\otimes p_1)u_A)^*] \\
=& [E_A\otimes p_1] \\
=& [E_A].
\end{align*}
We will next see that 
$\varphi_{A*} \circ \Ad(U_A)_*  = \tilde{\delta}_A\circ\varphi_{A*}.
$
We note that $K_0(C^*(G_A^a))$ is generated by the classes of projections
of the form
$T_{\bar{\xi}}T_{\bar{\xi}}^* \otimes S_\mu S_\mu^*$
where 
$\mu = (\mu_1,\dots,\mu_m) \in  B_k(\bar{X}_A),
\,
\bar{\xi} =(\xi_k,\dots,\xi_1) \in B_k(\bar{X}_{A^t})$
with $A(\xi_k,\mu_1) =1.$
We then have
\begin{align*}
 (\varphi_{A*} \circ \Ad(U_A)_*)
 ([T_{\bar{\xi}}T_{\bar{\xi}}^*\otimes S_\mu S_\mu^*]) 
= & [\varphi_A(T_{\xi_{k-1} \cdots\xi_1}T_{\xi_{k-1} \cdots\xi_1}^*
\otimes S_{\xi_k\mu} S_{\xi_k\mu}^*)] \\
= & [u_A^*(T_{\xi_{k-1} \cdots\xi_1}T_{\xi_{k-1} \cdots\xi_1}^*
\otimes S_{\xi_k\mu} S_{\xi_k\mu}^*)u_A] \\
= & [T_{\xi_{k-1} \cdots\xi_1}T_{\xi_{k-1} \cdots\xi_1}^*
\otimes S_{\xi_k\mu} S_{\xi_k\mu}^*]. 
\end{align*}
On the other hand,
\begin{align*}
\tilde{\delta}_A\circ\varphi_{A*}
([T_{\bar{\xi}}T_{\bar{\xi}}^*\otimes S_\mu S_\mu^*]) 
 = &
(\delta_{A^t}^{-1}\otimes\delta_A)([
u_A^*(T_{\bar{\xi}}T_{\bar{\xi}}^*\otimes S_\mu S_\mu^*)u_A]) \\ 
 = &
(\delta_{A^t}^{-1}\otimes\delta_A)([
T_{\bar{\xi}}T_{\bar{\xi}}^*\otimes S_\mu S_\mu^*]) \\ 
 = &
\delta_{A^t}^{-1}([T_{\bar{\xi}}T_{\bar{\xi}}^*])\otimes 
\delta_A([S_\mu S_\mu^*]). 
\end{align*}
As in \cite[Lemma 4.5]{MaMathScand1998}, 
$
\delta_{A^t}^{-1}([T_{\bar{\xi}}T_{\bar{\xi}}^*]) = 
[T_{\xi_{k-1} \cdots\xi_1}T_{\xi_{k-1} \cdots\xi_1}^*]
$
and
$\delta_A([S_\mu S_\mu^*])
= [S_{\xi_k\mu} S_{\xi_k\mu}^*].
$
Hence we have 
$$
 (\varphi_{A*} \circ \Ad(U_A)_*)
 ([T_{\bar{\xi}}T_{\bar{\xi}}^*\otimes S_\mu S_\mu^*])
=\tilde{\delta}_A\circ\varphi_{A*}([T_{\bar{\xi}}T_{\bar{\xi}}^*\otimes S_\mu S_\mu^*]). 
$$
\end{proof}
We note that the K-theoretic class $[E_A]$ of the projection $E_A$ has appeared in 
studying of K-theoretic duality by J. Kaminker--I. F. Putnam \cite{KamPut}.
\begin{lemma}
Let $A=[A(i,j)]_{i,j=1}^N$ and $B=[B(i,j)]_{i,j=1}^M$ 
be irreducible, non-permutation matrices with entries in $\{0, 1\}$.
Suppose that there exists an isomorphism
$\Phi:\WRA\otimes\K \longrightarrow \WRB\otimes\K$
of $C^*$-algebras such that 
\begin{equation}
\Phi\circ(\gamma^A_{(r,s)}\otimes\id)
 = (\gamma^B_{(r,s)}\otimes\id)\circ\Phi, \quad
(r,s) \in \T^2. \label{eq:4.1}
\end{equation}  
\begin{enumerate}
\renewcommand{\theenumi}{\roman{enumi}}
\renewcommand{\labelenumi}{\textup{(\theenumi)}}
\item
%%%%%%%%%%%%%%%%%%%
Then $\Phi$ induces an isomorphism
\begin{equation*}
\Phi_{0*}: K_0(\FTA)\otimes K_0(\FA) \longrightarrow K_0(\FTB)\otimes K_0(\FB) 
\end{equation*}
such that 
\begin{equation*}
\Phi_{0*} \circ(\delta_{A^t}^m\otimes\delta_A^n) 
=(\delta_{B^t}^m\otimes\delta_B^n)\circ\Phi_{0*}, 
\qquad (m,n) \in \Z^2.
\end{equation*}
%%%%%%%%%%%%%%%%%%%%%%%%%%%%
\item
There exist an $N^2\times M^2$-matrix $H,$
 an $M^2\times N^2$-matrix $K$
with entries in $\{0,1\}$ and a natural number $\ell$ such that 
\begin{gather}
(A^t\otimes A)^{\ell} =HK,\qquad
(B^t\otimes B)^{\ell} =KH, \label{eq:4.2}\\
(1\otimes A)H = H(1\otimes B), \qquad
K(1\otimes A) = (1\otimes B)K, \label{eq:4.3}\\
(A^t\otimes 1)H = H(B^t\otimes 1), \qquad
K(A^t\otimes 1) = (B^t\otimes 1)K. \label{eq:4.4}
\end{gather}
\end{enumerate}
\end{lemma}
\begin{proof}
(i)
Since 
$\Phi:\WRA\otimes\K \longrightarrow \WRB\otimes\K$
is an isomorphism of $C^*$-algebras satisfying \eqref{eq:4.1}, 
%\begin{equation*}
%\Phi\circ(\gamma^A_{(r,s)}\otimes\id)
%= (\gamma^B_{(r,s)}\otimes\id)\circ\Phi, \quad(r,s) \in \T^2, 
%\end{equation*}
it induces an isomorphism
\begin{equation*}
\Phi_1: (\WRA\otimes\K)\rtimes_{\gamma^A\otimes\id}\T^2
\longrightarrow
(\WRB\otimes\K)\rtimes_{\gamma^B\otimes\id}\T^2
\end{equation*}
of $C^*$-algebras of the crossed products.
Let $\hat{\gamma}^A, \hat{\gamma}^B$
be the dual actions on
$\WRA\rtimes_{\gamma^A}\T^2,
\WRB\rtimes_{\gamma^A}\T^2,
$
respectively.
By identifying 
$(\WRA\otimes\K)\rtimes_{\gamma^A\otimes\id}\T^2$ 
with
$(\WRA\rtimes_{\gamma^A}\T^2)\otimes\K,$
and
$(\WRB\otimes\K)\rtimes_{\gamma^B\otimes\id}\T^2$ 
with
$(\WRB\rtimes_{\gamma^B}\T^2)\otimes\K,$
we see that  
\begin{equation*}
\Phi_1\circ(\hat{\gamma}^A_{(m,n)}\otimes\id)
 = (\hat{\gamma}^B_{(m,n)}\otimes\id)\circ\Phi_1, \quad
(m,n) \in \Z^2. 
\end{equation*}
Hence
we have an isomorphism
\begin{equation*}
\Phi_{1*}: K_0(\WRA\rtimes_{\gamma^A}\T^2)
\longrightarrow
K_0(\WRB\rtimes_{\gamma^B}\T^2)
\end{equation*}
such that 
$$
\Phi_{1*}\circ\hat{\gamma}^A_{(m,n)*}=  
\hat{\gamma}^B_{(m,n)*}\circ\Phi_{1*}, \qquad (m,n) \in \Z^2.
$$
We then define
$\Phi_{0*}: K_0(\FTA)\otimes K_0(\FA)\longrightarrow 
K_0(\FTB)\otimes K_0(\FB)
$
by setting
\begin{equation*}
\Phi_{0*} = \varphi_{B*}\circ\iota_{B*}^{-1}\circ\Phi_{1*}\circ\iota_{A*}\circ\varphi_{A*}^{-1},=
\end{equation*}
where 
$\iota_{A*}: K_0(C^*(G_A^a)) =K_0(p_0(\WRA\rtimes_{\gamma^A}\T^2){p_0}) \longrightarrow
K_0(\WRA\rtimes_{\gamma^A}\T^2)
$ is the isomorphism in Lemma \ref{lem:4.1} (ii).
Hence the following diagram is commutative:
$$
\begin{CD}
K_0(\WRA\rtimes_{\gamma^A}\T^2) @>\Phi_{1*} >> K_0(\WRB\rtimes_{\gamma^B}\T^2) \\
@A{\iota_{A*} }AA  @AA{\iota_{B*}}A \\
K_0(p_0(\WRA\rtimes_{\gamma^A}\T^2){p_0}) @. K_0(p_0(\WRB\rtimes_{\gamma^B}\T^2)p_0) \\
@| @| \\
K_0(C^*(G_A^a)) @. K_0(C^*(G_A^a)) \\
@V{\varphi_{A*} }VV  @VV{\varphi_{B*}}V \\
K_0(\FTA)\otimes K_0(\FA)
@> \Phi_{0*} >> 
K_0(\FTA)\otimes K_0(\FA).
\end{CD}
$$
We then have 
\begin{align*}
& \Phi_{0*}\circ(\delta_{A^t}^m\otimes\delta_A^n) \\
=& (\varphi_{B*}\circ\iota_{B*}^{-1}\circ\Phi_{1*}\circ\iota_{A*}\circ\varphi_{A*}^{-1})
\circ
(\varphi_{A*}\circ\iota_{A*}^{-1}\circ\hat{\gamma}^A_{(m,n)*}\circ\iota_{A*}\circ\varphi_{A*}^{-1}\\
=& \varphi_{B*}\circ\iota_{B*}^{-1}\circ\Phi_{1*}\circ\hat{\gamma}^A_{(m,n)*}
\circ\iota_{A*}\circ\varphi_{A*}^{-1}\\
=& \varphi_{B*}\circ\iota_{B*}^{-1}\circ\hat{\gamma}^B_{(m,n)*}\circ\Phi_{1*}
\circ\iota_{A*}\circ\varphi_{A*}^{-1}\\
=& (\varphi_{B*}\circ\iota_{B*}^{-1}\circ\hat{\gamma}^B_{(m,n)*}\circ\iota_{B*}\circ\varphi_{B*}^{-1})
\circ
(\varphi_{B*}\circ\iota_{B*}^{-1}\circ\Phi_{1*}\circ\iota_{A*}\circ\varphi_{A*}^{-1})\\
=&(\delta_{B^t}^m\otimes\delta_B^n)\circ\Phi_{0*}.
\end{align*}

(ii)
By (i)
the isomorphism
$\Phi:\WRA\otimes\K \longrightarrow \WRB\otimes\K$
 satisfying \eqref{eq:4.1}
induces an isomorphism
\begin{equation*}
\Phi_{0*}: K_0(\FTA)\otimes K_0(\FA) \longrightarrow
K_0(\FTB)\otimes K_0(\FB) 
\end{equation*}
of ordered groups such that 
$\Phi_{0*}\circ(\delta_{A^t}^m\otimes\delta_A^n) 
=(\delta_{B^t}^m\otimes\delta_B^n)\circ\Phi_{0*}.$ 
Now 
\begin{equation*}
K_0(\FA)=\lim\{ \Z^N\overset{A^t}{\longrightarrow} \Z^N
\overset{A^t}{\longrightarrow} \cdots\}, \qquad
K_0(\FTA)=\lim\{ \Z^N\overset{A}{\longrightarrow} \Z^N
\overset{A}{\longrightarrow} \cdots\}
\end{equation*}
and
the dimension drop automorphisms
$\delta_A:K_0(\FA)\longrightarrow K_0(\FA)$
and
$\delta_{A^t}:K_0(\FTA)\longrightarrow K_0(\FTA)$
are defined by
$\delta_A([x,n]) = [x,n+1] (=[Ax,n])$ for $[x,n] \in K_0(\FA)$
and 
$\delta_{A^t}([y,n]) = [y,n+1] (=[A^ty,n])$ for $[y,n] \in K_0(\FTA),$
respectively
(\cite{CuntzInvent80}, \cite{CK}).
Since $K_0(\FTA)\otimes K_0(\FA) =K_0({\mathcal{F}}_{A^t\otimes A})$ 
and
$\delta_{A^t} \otimes\delta_A =\delta_{A^t\otimes A},$
we have an isomorphism of dimension triplets 
$$
(K_0({\mathcal{F}}_{A^t\otimes A}), \delta_{A^t\otimes A})
 \cong 
(K_0({\mathcal{F}}_{B^t\otimes B}), \delta_{B^t\otimes B})
$$
with dimension drop automorphisms.
Hence the two matrices
$A^t\otimes A$ and $B^t\otimes B$
are shift equivalent by \cite{Krieger1980}, 
which means that 
there exist an $N^2\times M^2$-matrix $H,$
 an $M^2\times N^2$-matrix $K$
with entries in nonnegative integers and a natural number $\ell$ such that 
\begin{gather*}
(A^t\otimes A)^{\ell} =HK,\qquad
(B^t\otimes B)^{\ell} =KH,\\
(A^t\otimes A)H = H(B^t\otimes B), \qquad
K(A^t\otimes A) = (B^t\otimes B)K.
\end{gather*}
Since
\begin{align*}
K_0(\FTA) \otimes K_0(\FA)
=& \lim\{ \Z^N\overset{A}{\longrightarrow} \Z^N
\overset{A}{\longrightarrow} \cdots\}
\otimes
\lim\{ \Z^N\overset{A^t}{\longrightarrow} \Z^N
\overset{A^t}{\longrightarrow} \cdots\} \\
\cong &
\lim\{ \Z^N\otimes\Z^N\overset{A\otimes A^t}{\longrightarrow} 
\Z^N\otimes\Z^N\overset{A\otimes A^t}{\longrightarrow} \cdots\}
\end{align*}
and similarly
\begin{equation*}
K_0(\FTB) \otimes K_0(\FTB)
\cong 
\lim\{ \Z^M\otimes\Z^M\overset{B\otimes B^t}{\longrightarrow} 
\Z^M\otimes\Z^M\overset{B\otimes B^t}{\longrightarrow} \cdots\},
\end{equation*}
$\Phi_{0*}$ induces an isomorphism
$$
\lim\{ \Z^N\otimes\Z^N\overset{A\otimes A^t}{\longrightarrow} 
\Z^N\otimes\Z^N\overset{A\otimes A^t}{\longrightarrow} \cdots\}
\longrightarrow
\lim\{ \Z^M\otimes\Z^M\overset{B\otimes B^t}{\longrightarrow} 
\Z^M\otimes\Z^M\overset{B\otimes B^t}{\longrightarrow} \cdots\},
$$
which satisfies
\begin{gather}
\Phi_{0*}\circ(\id \otimes\delta_A) 
= (\id \otimes\delta_B)\circ\Phi_{0*}, \label{eq:4.5} \\
\Phi_{0*}\circ(\delta_{A^t}\otimes\id) 
= (\delta_{B^t}\otimes\id\circ\Phi_{0*}. \label{eq:4.6}
\end{gather}
By the conditions \eqref{eq:4.5} and \eqref{eq:4.6},
we may take the matrices $H,K$ satisfying
\eqref{eq:4.2},\eqref{eq:4.3} and \eqref{eq:4.4}.
\end{proof}
For a matrix $A$, let us denote by
${\operatorname{Sp}}^{\times}(A)$
and
$\Spm(A)$
the set of nonzero eigenvalues of $A$ 
and the list of nonzero repeated eigenvalues of $A$ 
according to their multiplicity,
respectively.
\begin{lemma}\label{lem:4.4}
Let $A, B$ be irreducible, non-permutation matrices with entries in $\{0,1\}.$
Suppose that 
 there exist an $N^2\times M^2$-matrix $H,$
 an $M^2\times N^2$-matrix $K$
with entries in $\{0,1\}$ and a natural number $\ell$ satisfying
\eqref{eq:4.2},\eqref{eq:4.3} and \eqref{eq:4.4}.
Then 
\begin{equation}
\Spm(A^t\otimes A) =\Spm(B^t\otimes B)\quad
\text{ and }
 \quad
{\operatorname{Sp}}^{\times}(A)
= {\operatorname{Sp}}^{\times}(B). \label{eq:4.41}
\end{equation} 
\end{lemma}
\begin{proof}
We note that $\Spm(A) = \Spm(A^t).$
By the equalities \eqref{eq:4.3} and \eqref{eq:4.4},
we see that 
\begin{equation}
(A^t\otimes A)H = H(B^t\otimes B), \qquad
K(A^t\otimes A)= (B^t\otimes B)K. \label{eq:4.7}
\end{equation}
The equalities \eqref{eq:4.2} together with \eqref{eq:4.7}
show us that the matrices  
$A^t\otimes A$ and 
$B^t \otimes B$
are shift equivalent, so that 
\begin{equation}
\Spm(A^t\otimes A) =\Spm(B^t\otimes B) \label{eq:4.8} 
\end{equation}
by \cite[Theorem 7.4.10]{LM}.
For $\lambda \in {\operatorname{Sp}}^{\times}(A),$
one may take nonzero eigenvectors 
$u, v \in \mathbb{C}^N$ such that 
$Av = \lambda v$ and $A^tu =\lambda u$.
By \eqref{eq:4.3}, we have 
$$
(1\otimes B)K(u\otimes v) = K(1\otimes A)(u\otimes v)  
= K(u\otimes \lambda v)
= \lambda K(u\otimes v).
$$
By \eqref{eq:4.2}, we have 
$$
HK(u\otimes v) = (A^t\otimes A)^{\ell}(u\otimes v) = \lambda^\ell(u\otimes v)
$$
so that the vector $K(u\otimes v)$ is a nonzero eigenvector of the matrix $1\otimes B$
for the eigenvalue $\lambda.$
Hence 
$\lambda \in {\operatorname{Sp}}^{\times}(1\otimes B).$
Since 
${\operatorname{Sp}}^{\times}(1\otimes B)={\operatorname{Sp}}^{\times}( B),$
we have 
$\lambda \in {\operatorname{Sp}}^{\times}( B),$
so that 
${\operatorname{Sp}}^{\times}(A)
\subset {\operatorname{Sp}}^{\times}(B).
$
Similarly we have 
${\operatorname{Sp}}^{\times}(B)
\subset {\operatorname{Sp}}^{\times}(A)
$
and hence
${\operatorname{Sp}}^{\times}(A)
= {\operatorname{Sp}}^{\times}(B).
$
\end{proof}
%%%%%%%%%%%%%%%%%%%%%%%%%%%%%%%%%%%%%%%%%%%
\begin{lemma}\label{lem:4.5}
Suppose that  two irreducible, non-permutation matrices $A, B$ with entries in 
$\{0,1\}$ satisfy \eqref{eq:4.41}. 
Then we have
$$
\Spm(A) = \Spm(B) \quad
\text{ and hence }
\quad
\det(1 - A) = \det(1-B).
$$
\end{lemma}
\begin{proof}
Since both $A, B$ are irreducible, they have its periods as irreducible matrices,
which we denote by $p_A, p_B,$ respectively.
Since 
${\operatorname{Sp}}^{\times}(A)
= {\operatorname{Sp}}^{\times}(B),
$
their Perron-Frobenius eigenvalues coincide.
We denote the common eigenvalue by $\lambda_1$ 
which is positive.
There are exactly $p_A$ eigenvalues $\lambda$ of $\Spm(A)$ such that 
$|\lambda| = \lambda_1,$
so that we have $p_A = p_B$ which we denote by $p$.
Let $\omega$ be the $p$
th root  $e^{2\pi\sqrt{-1}\frac{1}{p}}$ of unity.
By  Perron-Frbobenius theorem for irreducible matrices,
one may find distinct eigenvalues 
$\{\lambda_1,\lambda_2,\dots,\lambda_L \} \subset
{\operatorname{Sp}}^{\times}(A)
(= {\operatorname{Sp}}^{\times}(B))
$
such that 
\begin{equation}
\lambda_1 > |\lambda_2|\ge |\lambda_3|\ge \cdots \ge |\lambda_L| \label{eq:lambdaineq}
\end{equation}
and the set
$\{ \omega^k \lambda_i\mid k=0,1,\dots,p-1, \, i=1,\dots, L\}$
is the full list of ${\operatorname{Sp}}^{\times}(A)
(={\operatorname{Sp}}^{\times}(B))$
(cf. \cite[Section 1.4]{Seneta}).
For each $i=1,\dots,L,$
the $p$ eigenvalues 
$$
\omega^k \lambda_i, \qquad  k=0,1,\dots,p-1
$$
have common multiplicities in $\Spm(A)$ and in  $\Spm(B)$,
respectively, which we denote by $m_i^A$ and $m_i^B,$ 
respectively.
Hence we know that $m_1^A = m_1^B =1.$ 
We put 
$\lambda_i(k) =\omega^k \lambda_i
$
for 
$k=0,1,\dots,p-1,\,  i=1,\dots,L.
$
Let $m_0^A, m_0^B$ be the multiplicities of zero eigenvalues 
of $A, B,$ respectively.
Then the characteristic polynomials of the matrices
$A^t\otimes A, B^t\otimes B$ are written such that 
\begin{gather*}
\varphi_{A^t\otimes A}(t) 
= t^{(m_0^A)^2}\prod_{i,j=1}^L\prod_{k,l=0}^{p-1}
(t-\lambda_i(k) \lambda_j(l))^{m_i^A m_j^A}, \\
\varphi_{B^t\otimes B}(t) 
= t^{(m_0^B)^2}\prod_{i,j=1}^L\prod_{k,l=0}^{p-1}
(t-\lambda_i(k) \lambda_j(l))^{m_i^B m_j^B}.
\end{gather*}
By the assumption
$\Spm(A^t\otimes A) =\Spm(B^t\otimes B),$
 we have
\begin{equation}
\prod_{i,j=1}^L\prod_{k,l=0}^{p-1}
(t-\lambda_i(k) \lambda_j(l))^{m_i^A m_j^A}
=\prod_{i,j=1}^L\prod_{k,l=0}^{p-1}
(t-\lambda_i(k) \lambda_j(l))^{m_i^B m_j^B}. \label{eq:4.51}
\end{equation}
The above polynomial of the left (resp. right) hand side is denoted by 
$\phi_A(t)$ (resp. $\phi_B(t)$).
%%%%%%%%%%%%%%%%%%%%%%%%%%%%%%
Suppose that 
\begin{equation*}
\lambda_1 \lambda_2 = \lambda_i(k) \lambda_j(l) \quad
\text{ for some } i, j=1,\dots, L  \text{ and }
k,l = 0,1,\dots,p-1.   
\end{equation*}
We may assume $i\le j.$
By the inequalities \eqref{eq:lambdaineq}
with
$|\lambda_i(k)| = |\lambda_i|, \, 
|\lambda_j(l)| = |\lambda_j|,
$
we have $ i=1,$
so that $\lambda_i(k) = \omega^k\lambda_1$.
Hence we have
$$
\lambda_2 = \omega^k \lambda_j(l) = \omega^{k+l}\lambda_j
$$
so that  $j =2$ and $k+l \equiv 0\,  (\operatorname{mod} p).$ 
We put $a = \lambda_1 \lambda_2.$
The power exponent of $(t- a)$ in the polynomial $\phi_A(t)$ is
\begin{equation*}
(m_1^A m_2^A +m_2^A m_1^A) \times 
|\{(k,l) \in \{0,1,\dots,p-1\}^2 \mid k+l \equiv 0\, (\operatorname{mod} p) \} |
=2 m_2^A p. 
\end{equation*} 
Similarly the power exponent of $(t- a)$ in the polynomial $\phi_B(t)$ 
is
$2 m_2^B p. $
Hence we have  
\begin{equation*}
m_2^A = m_2^B. 
\end{equation*}
Next assume that there exists $2\le h \le L$ such that 
\begin{equation}
m_n^A = m_n^B \quad \text{ for all } n \le h.  \label{eq:mnAB}
\end{equation}
Suppose that 
\begin{equation*}
\lambda_1 \lambda_{h+1} = \lambda_i(k) \lambda_j(l) \quad
\text{ for some } i, j=1,\dots, L  \text{ and } \,\,
k,l = 0,1,\dots,p-1.   
\end{equation*}
We may assume $i\le j.$
If $i=1$, 
then $\lambda_i(k) = \omega^k\lambda_1$.
Hence we have
$$
\lambda_{h+1} = \omega^k \lambda_j(l) = \omega^{k+l}\lambda_j
$$
so that  $j =h+1$ and $k+l \equiv 0\,  (\operatorname{mod} p).$ 
If $i \ne 1,$ we have $j < h+1$ because of the inequalities \eqref{eq:lambdaineq}.
We put
\begin{align*}
p_1(1, h+1) & =  \{ (i,j) \in \{2,\dots, L\}^2 \mid i <j, \, 
\lambda_1 \lambda_{h+1} = \lambda_i(k) \lambda_j(l) 
\text{ for some } k,l= 0,1,\dots, p-1  \},\\
p_0(1, h+1) & =  \{ i \in \{1,2,\dots, L\} \mid  
\lambda_1 \lambda_{h+1} = \lambda_i(k)^2 \text{ for some } k=0,1,\dots, p-1  \}.
\end{align*}
Both sets $p_1(1, h+1)$ and
$p_0(1, h+1)$ are possibly empty.
We note that 
$\lambda_i(k) \lambda_j(l) 
= \omega^{k+l} \lambda_i \lambda_j$
and
$
|\{ (k,l) \in \{1,\dots, p\}^2  \mid k+l \equiv 0\,  (\operatorname{mod} p) \} | =p.$
Put $ b = \lambda_1 \lambda_{h+1}.$ 
Hence the power exponent of $(t- b)$ in the polynomial $\phi_A(t)$ is
\begin{equation*}
2 m_1^A m_{h+1}^A p+
2( \sum_{(i,j) \in p_1(1,h+1)} m_i^A m_j^A)\cdot p 
+ \epsilon_p \sum_{i \in p_0(1,h+1)} m_i^A 
\end{equation*} 
where
$\epsilon_p =2$ if $ p$ is  even, 
and $\epsilon_p =1 $ if $ p $ is  odd. 
Similarly the power exponent of $(t- b)$ in the polynomial $\phi_B(t)$ is
\begin{equation*}
2 m_1^B m_{h+1}^B p +
2( \sum_{(i,j) \in p_1(1,h+1)} m_i^B m_j^B)\cdot p 
+ \epsilon_p \sum_{i \in p_0(1,h+1)} m_i^B 
\end{equation*} 
Any pair $(i,j) \in p_1(1,h+1)$ satisfies $i < j < h+1$
and
any element $i \in  p_0(1,h+1)$ satisfies $i < h+1.$
Hence the hypothesis \eqref{eq:mnAB} ensures that   
\begin{equation*}
m_{h+1}^A = m_{h+1}^B. 
\end{equation*}
%%%%%%%%%%%%%%%%%%%%%%%%%%%%%%%%%%%%%
Therefore we obtain that
$
\Spm(A) = \Spm(B).
$ 
Since
$$
\det(1 - A)= \prod_{i=1}^L\prod_{k=0}^{p-1} (1 - \lambda_i(k)^{m_i^A}), \qquad
\det(1 - B)= \prod_{i=1}^L\prod_{k=0}^{p-1} (1 - \lambda_i(k)^{m_i^B}), 
$$
the equality $\det(1 - A)  = \det(1 - B)$ follows from $\Spm(A) = \Spm(B)$. 
\end{proof}

%%%%%%%%%%%%%%%%%%%%%%%%%%%%%%
W. Parry--D. Sullivan in \cite{PS} proved that the determinant 
$\det(1-A)$ is invariant under flow equivalence of topological Markov shift $(\bar{X}_A, \bar{\sigma}_A).$
There is another crucial invariant of flow equivalence called the 
 Bowen--Franks group  written $\BF(A),$ 
which is defined by the abelian group 
$\Z^N/(1 -A)\Z^N$ 
for the $N\times N$ matrix $A$ with entries in $\{0,1\}$
(\cite{BF}).
J, Franks in \cite{Franks} proved that 
$\det(1-A)$  and $\BF(A)$ is a complete set of invariants of flow equivalence.
We note that the group $\BF(A)$ is isomorphic to the $K_0$-group
$K_0(\OA)$ of the Cuntz-Krieger algebra $\OA.$

We reach the following proposition.
\begin{proposition}\label{prop:4.5}
Assume that  $A$ and $B$ are irreducible, non-permutation matrices
with entries in $\{0,1\}$.
Suppose that 
there exists an isomorphism
$\Phi:\WRA\otimes\K \longrightarrow \WRB\otimes\K$
such that 
\begin{equation}
\Phi\circ(\gamma^A_{(r,s)}\otimes\id)
 = (\gamma^B_{(r,s)}\otimes\id)\circ\Phi, \quad
(r,s) \in \T^2. \label{eq:4.15}
\end{equation}  
Then the two-sided topological Markov shifts 
$(\bar{X}_B,\bar{\sigma}_B)$ and 
$(\bar{X}_A,\bar{\sigma}_A)$ are flow equivalent.
\end{proposition}
\begin{proof}
Suppose that 
there exists an isomorphism
$\Phi:\WRA\otimes\K \longrightarrow \WRB\otimes\K$
satisfying \eqref{eq:4.15}.
We then have 
$K_0(\WRA) = K_0(\WRB)$ so that 
$K_0(\OTA\otimes\OA)\cong
K_0(\OTB\otimes\OB)
$
and hence
$K_0(\OA) \cong K_0(\OB)$
by
K\"{u}nneth formulas.
This implies  that 
$\BF(A)$ is isomorphic to $\BF(B)$.
By the previous lemma, we have 
$\det(1-A) = \det(1-B)$.
Hence we conclude that 
$(\bar{X}_B,\bar{\sigma}_B)$ and 
$(\bar{X}_A,\bar{\sigma}_A)$ are flow equivalent
by Franks's theorem \cite{Franks}.
  \end{proof}

We will use Proposition \ref{prop:4.5} to prove Theorem \ref{thm:6.6} in Section 6.

%We will prove in the next section that the converse implication in Proposition \ref{prop:4.5}
%holds.

\medskip

\section{Gauge actions with potentials}
%%%%%%%%%%%%%%%%%%%%%%
In this section, we will define gauge actions $\gamma^{A,f}$ 
with potential function $f:\bar{X}_A\longrightarrow \Z$ 
on the $C^*$-algebra $\WRA$. 
For a continuous function $f \in C(\bar{X}_A,\Z)$ on $\bar{X}_A$
and $n \in \Z,$ we define a continuous function 
$f^n \in C(\bar{X}_A,\Z)$ by setting
\begin{equation*}
f^n(x)  = 
\begin{cases}
\sum_{i=0}^{n-1}f(\bar{\sigma}_A^i(x)) & \text{ for } n\ge 1, \\
0 & \text{ for } n=0, \\
- \sum_{i=n}^{-1}f(\bar{\sigma}_A^i(x)) & \text{ for } n\le -1.
 \end{cases}
\end{equation*}
It is easy to see that the identities
\begin{equation*}
f^{n+m}(x) = f^n(x) + f^m(\bar{\sigma}_A^n(x)), \qquad n,m\in \Z, \, x \in \bar{X}_A
\end{equation*}
hold.
For $f \in C(\bar{X}_A,\Z)$ 
and
$(x,p,q,y) \in G_A^{s,u}\rtimes\Z^2,$
define
\begin{align*}
\tilde{f}^+(x,p,q,y) 
& = \lim_{n\to\infty}\{ f^{n+p}(\bar{\sigma}_A(x)) -f^{n}(\bar{\sigma}_A(y))\}, \\
\tilde{f}^-(x,p,q,y) 
& = \lim_{n\to{-\infty}}\{ f^{n+q}(x) -f^{n}(y)\}. 
 \end{align*}
\begin{lemma}
Both $\tilde{f}^+, \tilde{f}^-: G_A^{s,u}\rtimes\Z^2\longrightarrow \Z$
are continuous groupoid homomorphisms from
$G_A^{s,u}\rtimes\Z^2$ to $\Z.$
\end{lemma}
\begin{proof}
Take an arbitrary point
$(x,p,q,y) \in G_A^{s,u}\rtimes\Z^2$
so that 
\begin{equation}
\lim_{n\to\infty} d(\bar{\sigma}_A^{n+p}(x),\bar{\sigma}_A^{n}(y)) =0,
\qquad
\lim_{n\to-\infty} d(\bar{\sigma}_A^{n+q}(x),\bar{\sigma}_A^{n}(y)) =0. \label{eq:5.0}
\end{equation}
By the first equality above,  we may find $N_1\in \N$ such that 
\begin{equation}
f(\bar{\sigma}_A^n(\bar{\sigma}_A^p(\bar{\sigma}_A(x)))) 
=f(\bar{\sigma}_A^n(\bar{\sigma}_A(y)))
\quad
\text{ for all } n \ge N_1. \label{eq:5.1}
\end{equation}
For $n\ge N_1,$ we have
\begin{align*}
 &f^{n+p}(\bar{\sigma}_A(x)) -f^{n}(\bar{\sigma}_A(y)) \\
=& f^{p}(\bar{\sigma}_A(x)) 
    +f^{n}(\bar{\sigma}_A^p(\bar{\sigma}_A(x))) -f^{n}(\bar{\sigma}_A(y)) \\
=& f^{p}(\bar{\sigma}_A(x)) 
   +f(\bar{\sigma}_A^p(\bar{\sigma}_A(x)))
   +f(\bar{\sigma}_A^{p+1}(\bar{\sigma}_A(x))) + \cdots 
   +f(\bar{\sigma}_A^{p+N_1 -1}(\bar{\sigma}_A(x))) \\
 & -f(\bar{\sigma}_A(y))-f(\bar{\sigma}_A^2(y))-\cdots 
    -f(\bar{\sigma}_A^{N_1-1}(\bar{\sigma}_A(y))) \\
= & f^{p}(\bar{\sigma}_A(x))  + f^{N_1}(\bar{\sigma}_A^p(\bar{\sigma}_A(x)))
     -f^{N_1}(\bar{\sigma}_A(y))
\end{align*}
so that 
\begin{equation}
\tilde{f}^+(x,p,q,y)
=  f^{p+N_1}(\bar{\sigma}_A(x))  %+ f^{N_1}(\bar{\sigma}_A^p(\bar{\sigma}_A(x)))
     -f^{N_1}(\bar{\sigma}_A(y)). \label{eq:5.2}
\end{equation}
By the second equality of \eqref{eq:5.0}, 
we may similarly find a negative integer $N_2 \in \Z$ such that 
\begin{equation}
\tilde{f}^-(x,p,q,y)
=  f^{q+N_2}(x)  %+ f^{N_2}(\bar{\sigma}_A^q(x))
     -f^{N_2}(y). \label{eq:5.3}
\end{equation}
Hence both the values 
$\tilde{f}^+(x,p,q,y), \tilde{f}^-(x,p,q,y)$
are defined.

For $(x,p,q,y), (x', p',q',y') \in G_A^{s,u}\rtimes\Z^2$ with $x' =y$, we have
\begin{align*}
& \tilde{f}^+((x,p,q,y)(x', p',q',y')) \\
= & \tilde{f}^+(x,p + p', q+q',y') \\
= & \lim_{n\to\infty}\{ f^{n+p+p'}(\bar{\sigma}_A(x)) - f^{n}(\bar{\sigma}_A(y')) \} \\
= & \lim_{n\to\infty}\{ 
f^{n+p}(\bar{\sigma}_A(x)) +f^{p'}(\bar{\sigma}_A^n(\bar{\sigma}_A^p(\bar{\sigma}_A(x)))) 
- f^{n}(\bar{\sigma}_A(y')) \}.
\end{align*}
On the other hand, we have 
\begin{align*}
& \tilde{f}^+(x,p,q,y) + \tilde{f}^+(x', p',q',y')) \\
= &
 \lim_{n\to\infty}\{ f^{n+p}(\bar{\sigma}_A(x)) -f^{n}(\bar{\sigma}_A(y))\}
+ \lim_{n\to\infty}\{ f^{n+p'}(\bar{\sigma}_A(x')) -f^{n}(\bar{\sigma}_A(y'))\} \\
= &
 \lim_{n\to\infty}\{ 
f^{n+p}(\bar{\sigma}_A(x))  +f^{n+p'}(\bar{\sigma}_A(y))
-f^{n}(\bar{\sigma}_A(y))  -f^{n}(\bar{\sigma}_A(y'))\} \\
= & \lim_{n\to\infty}\{ 
f^{n+p}(\bar{\sigma}_A(x)) +f^{p'}(\bar{\sigma}_A^n(\bar{\sigma}_A(y))) 
- f^{n}(\bar{\sigma}_A(y')) \}.
\end{align*}
By \eqref{eq:5.1}, 
the equality
\begin{equation*}
\lim_{n\to\infty} f^{p'}(\bar{\sigma}_A^n(\bar{\sigma}_A^p(\bar{\sigma}_A(x)))) 
=
\lim_{n\to\infty}
f^{p'}(\bar{\sigma}_A^n(\bar{\sigma}_A(y)))
\end{equation*}
holds, so that 
\begin{equation*}
 \tilde{f}^+((x,p,q,y)(x', p',q',y')) 
=\tilde{f}^+(x,p,q,y) + \tilde{f}^+(x', p',q',y')
\end{equation*}
and similarly
\begin{equation*}
 \tilde{f}^-((x,p,q,y)(x', p',q',y')) 
=\tilde{f}^-(x,p,q,y) + \tilde{f}^-(x', p',q',y').
\end{equation*}
The identities
\begin{equation*}
 \tilde{f}^+((x,p,q,y)^{-1}) 
=-\tilde{f}^+(x,p,q,y),
\qquad 
\tilde{f}^-((x,p,q,y)^{-1}) 
=-\tilde{f}^-(x,p,q,y)
\end{equation*}
are easily seen.
As the continuity of 
$\tilde{f}^+, \tilde{f}^-$
follows from the formulas \eqref{eq:5.2}, \eqref{eq:5.3},
we know that they are 
 continuous groupoid homomorphisms from
$G_A^{s,u}\rtimes\Z^2$ to $\Z.$
\end{proof}
Define 
a continuous groupoid homomorphism
$\tilde{f}:
G_A^{s,u}\rtimes\Z^2\longrightarrow \Z$
by 
$\tilde{f}(x,p,q,y) =\tilde{f}^+(x,p,q,y) -\tilde{f}^-(x,p,q,y).$
Recall that the $C^*$-algebra $\WRA$
is represented on the Hilbert $C^*$-right module 
$\ell^2(G_A^{s,u}\rtimes\Z^2)$ over 
$C_0((G_A^{s,u}\rtimes\Z^2)^\circ) (= C(\bar{X}_A))$
as the reduced groupoid $C^*$-algebra. 
For $f \in C(\bar{X}_A,\Z),$
$(r,s) \in \T^2$ and $\xi \in \ell^2(G_A^{s,u}\rtimes\Z^2),$
we set
\begin{align*}
[U_s(\tilde{f}^+)\xi](x,p,q,y)
& = \exp\{ 2\pi\sqrt{-1}\tilde{f}^+(x,p,q,y)s\} \xi(x,p,q,y), \\
[U_r(\tilde{f}^-)\xi](x,p,q,y)
& = \exp\{ 2\pi\sqrt{-1}\tilde{f}^-(x,p,q,y)r\} \xi(x,p,q,y), \\
U_{(r,s)}(\tilde{f}) &= U_r(\tilde{f}^-)U_{s}(\tilde{f}^+).
\end{align*}
Since
$\tilde{f}^+, \tilde{f}^-$ and $\tilde{f}$
are groupoid homomorphisms
from $G_A^{s,u}\rtimes\Z^2$ to $\Z$,
the operators $U(\tilde{f}^+), U(\tilde{f}^-)$
yield unitary representations of $\T$ 
and $U(\tilde{f})$ does a unitary representation of $\T^2$.

\begin{proposition}
For $f \in C(\bar{X}_A,\Z)$, the correspondence
$a \in \WRA \longrightarrow 
\Ad(U_{(r,s)}(\tilde{f}))(a) ( = U_{(r,s)}(\tilde{f})aU_{(r,s)}(\tilde{f})^*) \in \WRA $
defines an automorphism on $\WRA$ such that 
$(r,s) \in \T^2 \longrightarrow \Ad(U_{(r,s)}) \in  \Aut(\WRA)$
 gives rise to an action of $\T^2$ on $\WRA$
and
its restriction to the subalgebra $C(\bar{X}_A)$ is identity.   
\end{proposition}
\begin{proof}
For $a \in C_c(G_A^{s,u}\rtimes\Z^2),$
$\xi \in \ell^2(G_A^{s,u}\rtimes\Z^2),$
$(x,p,q,y) \in G_A^{s,u}\rtimes\Z^2$, we have
\begin{align*}
& [\Ad(U_{(r,s)}(\tilde{f}))(a)\xi](x,p,q,y) \\
%= &U_{(r,s)}(\tilde{f}))a U_{(r,s)}(\tilde{f}))^*\xi](x,p,q,y) \\
= &\exp\{ 2\pi\sqrt{-1}( \tilde{f}^+(x,p,q,y)s + \tilde{f}^-(x,p,q,y)r) \} 
[a U_{(r,s)}(\tilde{f})^*\xi](x,p,q,y). 
\end{align*}
Now the equalities  
\begin{align*}
& [a U_{(r,s)}(\tilde{f})^*\xi](x,p,q,y) \\
= & \sum_{\gamma; r(\gamma) = x} 
a(\gamma)[ U_{(r,s)}(-\tilde{f})\xi](\gamma^{-1}\cdot (x,p,q,y)) \\
= & \sum_{\gamma; r(\gamma) = x} 
a(\gamma) \cdot \exp\{ 2 \pi \sqrt{-1}(\tilde{f}^+(\gamma)s +\tilde{f}^-(\gamma)r)\}\\
 & \cdot \exp\{ -2\pi\sqrt{-1}( \tilde{f}^+(x,p,q,y)s + \tilde{f}^-(x,p,q,y)r) \}
\xi(\gamma^{-1}\cdot (x,p,q,y))  
\end{align*}
hold, so that 
\begin{align*}
& [\Ad(U_{(r,s)}(\tilde{f}))(a)\xi](x,p,q,y) \\
= & \sum_{\gamma; r(\gamma) = x} 
a(\gamma) \cdot \exp\{ 2 \pi \sqrt{-1}(\tilde{f}^+(\gamma)s +\tilde{f}^-(\gamma)r)\}
\xi(\gamma^{-1}\cdot (x,p,q,y)). 
\end{align*}
Let us identify
$U_{(r,s)}(\tilde{f})$ with  the  continuous function on 
$G_A^{s,u}\rtimes\Z^2$
defined by
$$
U_{(r,s)}(\tilde{f})(\gamma)
 = \exp\{ 2 \pi \sqrt{-1}(\tilde{f}^+(\gamma)s +\tilde{f}^-(\gamma)r)\},
\qquad \gamma \in G_A^{s,u}\rtimes\Z^2.
$$
Hence we have 
\begin{align*}
 [\Ad(U_{(r,s)}(\tilde{f}))(a)\xi](x,p,q,y) 
= & \sum_{\gamma; r(\gamma) = x} 
(U_{(r,s)}(\tilde{f})\cdot a )(\gamma)
\xi(\gamma^{-1}\cdot (x,p,q,y)) \\
= & 
[(U_{(r,s)}(\tilde{f})\cdot a )\xi](x,p,q,y)
\end{align*}
so that 
\begin{equation}
\Ad(U_{(r,s)}(\tilde{f}))(a) = U_{(r,s)}(\tilde{f})\cdot a \qquad 
\text{ for } a \in C_c(G_A^{s,u}\rtimes\Z^2) \label{eq:5.4}
\end{equation}
where $U_{(r,s)}(\tilde{f})\cdot a $ is the pointwise product between the two functions
$U_{(r,s)}(\tilde{f})$ and $ a.$ 
Thus  
$\Ad(U_{(r,s)}(\tilde{f}))(a)$ belongs to the algebra 
$C_c(G_A^{s,u}\rtimes\Z^2),$
so that 
$\Ad(U_{(r,s)}(\tilde{f}))$ yields an automorphism of the $C^*$-algebra $\WRA.$

Especially for a continuous function
$a \in C(\bar{X}_A)$
on $\bar{X}_A$, it is regarded as an element of 
$C_c(G_A^{s,u}\rtimes\Z^2)$ by 
\begin{equation*}
a(x,p,q,y) = 
\begin{cases}
a(x) & \text{ if } x=y, \,  p=q=0, \\
0 & \text{ otherwise.}
\end{cases}
\end{equation*}
For $x=y, \, p=q=0$, we know that 
$\tilde{f}^+(x,p,q,y) =\tilde{f}^-(x,p,q,y) = 0$
so that  
\begin{equation*}
\Ad(U_{(r,s)}(\tilde{f}))(a) = U_{(r,s)}(\tilde{f}))\cdot a =a \qquad 
\text{ for } a \in C_c(\bar{X}_A).
\end{equation*}
\end{proof}
We denote by $\gamma^{A,f}_{(r,s)}$ 
the automorphism $\Ad(U_{(r,s)}(\tilde{f}))$ on $\WRA,$
which yields an action called gauge action with potential function $f$,
or weighted gauge action.
For the constant function $f \equiv 1$,
we have 
$\tilde{f}^+(x,p,q,y) = p, \, \tilde{f}^-(x,p,q,y) = q$
so that the action 
$\gamma^{A,f}_{(r,s)}$ for $f\equiv 1$
coincides with the previously defined  action 
$\gamma^{A}_{(r,s)}.$

Let $U_j(0), j=1,2,\dots,N$ be the cylinder sets on $\bar{X}_A$  such that 
$$  
U_j(0)  =\{ (x_n)_{n \in \Z} \in \bar{X}_A \mid x_0 = j \}.
$$
Let $\chi_{U_j(0)}$ be the characteristic function on $\bar{X}_A$ of the cylinder set
$U_j(0).$
\begin{lemma}\label{lem:5.3}
Suppose that 
$f = \sum_{j=1}^N f_j \chi_{U_j(0)}$ for some integers $f_j \in \Z$.
Then we have
\begin{equation*}
\gamma^{A,f}_{(r,s)} = \alpha^{A^t,f}_r \otimes \alpha^{A,f}_s
\qquad 
\text{ on }
\WRA = E_A(\OTA\otimes \OA)E_A,
\end{equation*}
where
$\alpha^{A^t,f}_r \in \Aut(\OTA),
\alpha^{A,f}_s \in \Aut(\OA)$ are defined by
\begin{align*}
\alpha^{A^t,f}_r(T_j) & = 
\exp\{2\pi\sqrt{-1}f_j r\} T_j, \qquad j=1,2,\dots, N, \\
\alpha^{A,f}_s(S_j) & = 
\exp\{2\pi\sqrt{-1}f_j s\} S_j, \qquad j=1,2,\dots, N.
\end{align*}
\end{lemma}
\begin{proof}
For $m,n,k,l \in \N$ and
\begin{gather*}
 \mu =(\mu_1,\dots,\mu_m), \quad \nu =(\nu_1,\dots,\nu_n) \in B_*(\bar{X}_A),\\
 \bar{\xi} =(\xi_k,\dots,\xi_1), \quad \bar{\eta} =(\eta_l,\dots,\eta_1) \in B_*(\bar{X}_{A^t})
\end{gather*}
satisfying
$ 
 A(\xi_k,\mu_1) = A(\eta_l,\nu_1) =1,
$
we write
\begin{align*}
U_{\xi\mu,\eta\nu}=\{
&(x,m-n,l-k,y) \in G_A^{s,u}\rtimes\Z^2\mid 
(\bar{\sigma}_A^m(x), \bar{\sigma}_A^n(y)) \in G_A^{s,0}, 
(\bar{\sigma}_A^k(x), \bar{\sigma}_A^l(y)) \in G_A^{u,0}, \\
& x_{[1,m]} = \mu,  \, y_{[1,n]} = \nu, \,
   x_{[-k+1,0]} = \xi, \, y_{[-l+1,0]} = \eta \}
\end{align*}
where
\begin{gather*}
G_A^{s,0} = \{(x,y) \in \bar{X}_A\times\bar{X}_A \mid x_i = y_i \text{ for all } i \in \Zp\}, \\
G_A^{u,0} = \{(x,y) \in \bar{X}_A\times\bar{X}_A 
\mid x_{-i} = y_{-i} \text{ for all } i \in \Zp\}.
\end{gather*}
As in \cite[Section 9]{MaPre2017a}, 
the correspondence 
$
\chi_{U_{\xi\mu,\eta\nu}}
\longleftrightarrow 
T_{\bar{\xi}} T_\eta^*\otimes S_\mu S_\nu^*
$
gives rise to an isomorphism between
the groupoid $C^*$-algebra $C^*(G_A^{s,u}\rtimes\Z^2)$
and the algebra $\WRA$.
For $(x,p,q,y) \in U_{\xi\mu,\eta\nu}$
with $p = m-n, q= l-k$, one may take 
$N_1 = n$ so that
\begin{align*}
 \tilde{f}^+(x,p,q,y) 
= & 
f^{p}(\bar{\sigma}_A(x)) 
   +f(\bar{\sigma}_A^p(\bar{\sigma}_A(x)))
   +f(\bar{\sigma}_A^{p+1}(\bar{\sigma}_A(x))) +\cdots 
   +f(\bar{\sigma}_A^{p+N_1-1}(\bar{\sigma}_A(x))) \\
 & -f(\bar{\sigma}_A(y))-f(\bar{\sigma}_A^2(y))-\cdots -f(\bar{\sigma}_A^{N_1}(y)) \\
= & f(\bar{\sigma}_A(x))  + f(\bar{\sigma}_A(\bar{\sigma}_A(x)))+\cdots+
f(\bar{\sigma}_A^m(x)) \\
 &    -f(\bar{\sigma}_A(y))
-f(\bar{\sigma}_A^2(y))-\cdots
-f(\bar{\sigma}_A^n(y)) \\
=& (f_{\mu_1}  +f_{\mu_2}  +\cdots+f_{\mu_m} ) 
-( f_{\nu_1}  +f_{\nu_2}  +\cdots+f_{\nu_n})
\end{align*}
because ${\bar{\sigma}_A(x)}_{[0,m-1]} =\mu,$
${\bar{\sigma}_A(y)}_{[0,n-1]} =\nu,$
and similarly
\begin{align*}
 \tilde{f}^-(x,p,q,y) 
= & f(x) + f(\bar{\sigma}_A^{-1}(x))  + f(\bar{\sigma}_A^{-2}(x))+\cdots+
f(\bar{\sigma}_A^{-k+1}(x)) \\
 &    -f(y) - f(\bar{\sigma}_A^{-1}(y))
-f(\bar{\sigma}_A^{-2}(y)) -\cdots
-f(\bar{\sigma}_A^{-l+1}(y)) \\
=& (f_{\xi_1}  +f_{\xi_2}  +\cdots+f_{\xi_k})  
-( f_{\eta_1}  +f_{\eta_2}  +\cdots+f_{\eta_l}).
\end{align*}
It then follows that by \eqref{eq:5.4}
\begin{align*}
& [\Ad(U_{(r,s)}(\tilde{f}))(\chi_{U_{\xi\mu,\eta\nu}})](x,p,q,y) \\
= &[U_{(r,s)}(\tilde{f}))\cdot \chi_{U_{\xi\mu,\eta\nu}})](x,p,q,y) \\
= &\exp\{ 2\pi\sqrt{-1}( \tilde{f}^+(x,p,q,y)s + \tilde{f}^-(x,p,q,y)r) \} 
\chi_{U_{\xi\mu,\eta\nu}}(x,p,q,y)  
\end{align*}
so that 
\begin{align*}
 &\Ad(U_{(r,s)}(\tilde{f}))(\chi_{U_{\xi\mu,\eta\nu}}) \\
= &\exp[ 2\pi\sqrt{-1}
\{(f_{\mu_1}  +f_{\mu_2}  +\cdots+f_{\mu_m})  
-( f_{\nu_1}  +f_{\nu_2}  +\cdots+f_{\nu_n})\}s \\
 & +\{ (f_{\xi_1}  +f_{\xi_2}  +\cdots+f_{\xi_k})  
-( f_{\eta_1}  +f_{\eta_2}  +\cdots+f_{\eta_l})\}r]
\cdot 
\chi_{U_{\xi\mu,\eta\nu}},
\end{align*}
proving that the equality
\begin{equation*}
\gamma^{A,f}_{(r,s)}(\chi_{U_{\xi\mu,\eta\nu}})
 = \alpha^{A^t,f}_r(T_\xi T_\eta^*) \otimes \alpha^{A,f}_s(S_\mu S_\nu^*).
\end{equation*}
\end{proof}

\begin{proposition}\label{prop:5.4}
Let $\varphi:\bar{X}_A\longrightarrow \bar{X}_B$ 
be a topological conjugacy between two-sided topological Markov shifts
$(\bar{X}_A, \bar{\sigma}_A)$ and $(\bar{X}_B, \bar{\sigma}_B).$
Suppose that $f \in C(\bar{X}_A,\Z)$ and $g \in C(\bar{X}_B,\Z)$
satisfy $f = g \circ\varphi.$  
Then there exists an isomorphism $\Phi:\WRA\longrightarrow\WRB$
of $C^*$-algebras such that 
\begin{equation*}
\Phi(C(\bar{X}_A)) = C(\bar{X}_B),
\qquad
\Phi\circ\gamma^{A,f}_{(r,s)} = \gamma^{B,g}_{(r,s)}\circ\Phi,
\quad (r,s) \in \T^2.
\end{equation*}
\end{proposition}
\begin{proof}
The topological conjugacy
$\varphi:\bar{X}_A\longrightarrow \bar{X}_B$ 
induces an isomorphism 
$\tilde{\varphi}: G_A^{s,u}\rtimes\Z^2\longrightarrow G_B^{s,u}\rtimes\Z^2$
of \'etale groupoids such that 
$\tilde{\varphi}(x,p,q,y) =(\varphi(x),p,q,\varphi(y))$
for $(x,p,q,y) \in G_A^{s,u}\rtimes\Z^2.$
It gives rise to a unitary
written 
$V_\varphi: \ell^2(G_B^{s,u}\rtimes\Z^2)
\longrightarrow \ell^2(G_A^{s,u}\rtimes\Z^2)$
satisfying 
$V_\varphi(\xi) = \xi\circ\tilde{\varphi}$ 
for $\xi \in \ell^2(G_B^{s,u}\rtimes\Z^2).$
Assume that the $C^*$-algebras 
$\WRA$ and $\WRB$ are represented on 
$\ell^2(G_A^{s,u}\rtimes\Z^2)$ and $\ell^2(G_A^{s,u}\rtimes\Z^2)$
as reduced groupoid $C^*$-algebras.
Since 
$V_\varphi^* a V_\varphi = a \circ\tilde{\varphi}^{-1} \in C_c(G_B^{s,u}\rtimes\Z^2)
$
for 
$a \in C_c(G_A^{s,u}\rtimes\Z^2),$
we know that $\Phi(a) = V_\varphi^* a V_\varphi, a \in \WRA$
gives rise to an isomorphism $\WRA\longrightarrow\WRB$
of $C^*$-algebras.
Since $\varphi:\bar{X}_A\longrightarrow \bar{X}_B$
is a topological conjugacy,
we know that 
\begin{equation*}
\tilde{g}^+(\varphi(x),p,q,\varphi(y))= \tilde{f}^+(x,p,q,y),
\qquad
\tilde{g}^-(\varphi(x),p,q,\varphi(y))= \tilde{f}^-(x,p,q,y)
\end{equation*} 
for $(x,p,q,y) \in G_A^{s,u}\rtimes\Z^2$
so that  for $\xi\in \ell^2(G_A^{s,u}\rtimes\Z^2),
(x,p,q,y) \in G_A^{s,u}\rtimes\Z^2,$
we have
\begin{align*}
& [V_\varphi U_{(r,s)}(\tilde{g}) V_\varphi^*\xi](x,p,q,y) \\
= & [U_{(r,s)}(\tilde{g}) V_\varphi^*\xi](\varphi(x),p,q,\varphi(y)) \\
=& \exp
\{ 2\pi\sqrt{-1}( \tilde{g}^+(\varphi(x),p,q,\varphi(y)) s + \tilde{g}^-(\varphi(x),p,q,\varphi(y))r) \} 
[V_\varphi^*\xi](\varphi(x),p,q,\varphi(y)) \\
=& \exp\{ 2\pi\sqrt{-1}( \tilde{f}^+(x,p,q,y) s + \tilde{f}^-(x,p,q,y)r) \}
\xi(x,p,q,y) \\
=& [U_{(r,s)}(\tilde{f})\xi](x,p,q,y). 
\end{align*}
Hence we have
$
V_\varphi U_{(r,s)}(\tilde{g}) V_\varphi^* = U_{(r,s)}(\tilde{f}),
$
%As $$
%V_\varphi^* U_{(r,s)}(\tilde{f}) V_\varphi = U_{(r,s)}(\tilde{g}) 
%\qquad \text{ for } (r,s) \in \T^2,
%$$
so that the equality
$
\Phi\circ\gamma^{A,f}_{(r,s)} = \gamma^{B,g}_{(r,s)}\circ\Phi
$
holds.
Since 
$a \circ\tilde{\varphi}^{-1} \in C_c({(G_B^{s,u}\rtimes\Z^2)}^\circ))
$
for
$a\in C_c({(G_A^{s,u}\rtimes\Z^2)}^\circ)),$
the equality
$\Phi(C(\bar{X}_A)) = C(\bar{X}_B)
$ is obvious
(cf.  \cite[Theorem 1.1]{MaPre2017b}).
\end{proof}

\begin{corollary}\label{cor:5.5}
Let $B$ be an $M\times M$ irreducible non-permutation matrix with 
entries in $\{0,1\}$.
For any continuous function $g \in C(\bar{X}_B,\Z)$ on $\bar{X}_B,$
there exist an $N\times N$ irreducible non-permutation matrix $A$ with 
entries in $\{0,1\}$
and a continuous function    
$f = \sum_{j=1}^N f_j \chi_{U_j(0)}$ for some integers $f_j \in \Z$
such that 
$(\bar{X}_A, \bar{\sigma}_A)$ is topologically conjugate to
$(\bar{X}_B, \bar{\sigma}_B)$
and there exists an isomorphism
$\Phi: \WRA\longrightarrow \WRB$ such that 
\begin{equation*}
\Phi(C(\bar{X}_A)) = C(\bar{X}_B),
\qquad
\Phi\circ\gamma^{A,f}_{(r,s)} = \gamma^{B,g}_{(r,s)}\circ\Phi,
\quad (r,s) \in \T^2.
\end{equation*}
\end{corollary}
\begin{proof}
There exists $K \in \N$ such that 
$g  = \sum_{\mu \in B_K(\bar{X}_B)} g_\mu \chi_{U_\mu}$
for some $g_\mu \in \Z$
where $U_\mu$ is the cylinder set of $\bar{X}_B$ for a word $\mu \in B_K(\bar{X}_B)$.
By taking $K$-higher block representation of  $\bar{X}_B$
and its $K$ higher block matrix of $B$ as $A$ (see \cite[1.4]{LM}), 
and shifting $g$, 
one may have a topological conjugacy
$\varphi: \bar{X}_A \longrightarrow \bar{X}_B$
and a continuous function 
$f = \sum_{j=1}^N f_j \chi_{U_j(0)}$ for some integers $f_j \in \Z$
such that 
$f = g \circ\varphi$.
Hence we get the desired assertion by Proposition \ref{prop:5.4}.
\end{proof}

%%%%%%%%%%%%%%%%%%%%%%%%%%%%%%%%%%%%%%%%%%%%%%%%%%%%%%%%%%
%%%%%%%%%%%%%%%%%%%%%%%%%%%%%%%%%%%%%%%%%%%%%%%%%%%%%%%%%
\section{Flow equivalence}
%%%%%%%%%%%%%%%%%%%%%%%%%%%%%%%%%%%%%%%%%%%%%%%%
We fix an irreducible, non-permutation matrix $A.$
 Let $f: \bar{X}_A\longrightarrow \N$ 
be a continuous function  on $\bar{X}_A$ such that 
$f = \sum_{j=1}^N f_j \chi_{U_j(0)}$ for some integers $f_j \in \N$.
Put $m_j  =f_j -1$ for $j=1,\dots,N.$
Consider the new graph 
$\G_f =(\V_f,\E_f)$ with its transition matrix $A_f$  from the graph
$\G =(\V,\E)$ for the matrix $A$
defined in \eqref{eq:2.3} in Section 2.
The vertex set $\V_f$ is
$\cup_{j=1}^N \{ j_0, j_1, j_2,\dots, j_{m_j} \}$
which is denoted by $\tilde{\Sigma}$, and
if $A(j,k) =1,$ then
\begin{equation*}
A_f(j_0,j_1) =A_f(j_1,j_2) =\cdots =A_f(j_{mj-1},j_{m_j}) =A_f(j_{m_j}, k_0) =1.
\end{equation*}
%Hence the size of the matrix $A_f$ is
%$(f_1+f_2+ \cdots +f_N) \times (f_1+f_2+ \cdots +f_N). $ 
Let us denote by 
\begin{equation*}
\tilde{S}_{j_0}, \tilde{S}_{j_1},\tilde{S}_{j_2},\dots, \tilde{S}_{j_{m_j}} 
\quad\text{ and }\quad
\tilde{T}_{j_0}, \tilde{T}_{j_1},\tilde{T}_{j_2},\dots, \tilde{T}_{j_{m_j}}
\end{equation*}
the canonical generating partial isometries of $\OAf$ and $\OTAf$ 
respectively which satisfy
\begin{gather}
\sum_{j=1}^N (\tilde{S}_{j_0}\tilde{S}_{j_0}^* +  \tilde{S}_{j_1}\tilde{S}_{j_1}^* +\cdots+
 \tilde{S}_{j_{m_j}} \tilde{S}_{j_{m_j}}^*) =1, \label{eq:4.S1}\\ 
%\tilde{S}_j^*\tilde{S}_j= \tilde{S}_{j_1}\tilde{S}_{j_1}^*, \qquad
\tilde{S}_{j_n}^*\tilde{S}_{j_n}= \tilde{S}_{j_{n+1}}\tilde{S}_{j_{n+1}}^*, \qquad n=0, 1,\dots,m_j-1,
\label{eq:4.S2}\\
\tilde{S}_{j_{m_j}}^* \tilde{S}_{j_{m_j}} = \sum_{k=1}^N A(j,k)\tilde{S}_{k_0} \tilde{S}_{k_0}^*, \qquad
j=1,2,\dots,N \label{eq:4.S3}
\end{gather}
and
\begin{gather}
\sum_{j=1}^N (\tilde{T}_{j_0}\tilde{T}_{j_0}^* +  \tilde{T}_{j_1}\tilde{T}_{j_1}^* +\cdots+
 \tilde{T}_{j_{m_j}} \tilde{T}_{j_{m_j}}^*) =1, \label{eq:4.T1}\\ 
%\tilde{T}_{j_1}^*\tilde{T}_{j_1}= \tilde{T}_j\tilde{T}_j^*, \qquad
\tilde{T}_{j_{n+1}}^*\tilde{T}_{j_{n+1}}= \tilde{T}_{j_n}\tilde{T}_{j_n}^*, \qquad n=0, 1,\dots,m_j-1,
\label{eq:4.T2}\\
\tilde{T}_{j_0}^* \tilde{T}_{j_0} 
= \sum_{k=1}^N A^t(j,k)\tilde{T}_{k_{m_k}} \tilde{T}_{k_{m_k}}^*, \qquad
j=1,2,\dots,N. \label{eq:4.T3}
\end{gather}
We set
\begin{equation*}
S_j= \tilde{S}_{j_0}\tilde{S}_{j_1}\tilde{S}_{j_2}\cdots \tilde{S}_{j_{m_j}} 
\quad\text{ and }\quad
T_j = \tilde{T}_{j_{m_j}}\tilde{T}_{j_{m_j-1}}\cdots \tilde{T}_{j_1} \tilde{T}_{j_0}
\quad\text{ for }\quad j=1,\dots,N.
\end{equation*}
Define the projections
\begin{equation*}
P_A = \sum_{j=1}^N \tilde{S}_{j_0} \tilde{S}_{j_0}^*\quad\text{ and }\quad
P_{A^t} = \sum_{j=1}^N \tilde{T}_{j_{m_j}}\tilde{T}_{j_{m_j}}^*.
\end{equation*}
We denote by 
$C^*(S_1, \dots, S_N)$ (resp.  $C^*(T_1, \dots, T_N))$
the $C^*$-subalgebra of $\OAf$ (resp. $\OTAf$) generated by
$ S_1, \dots, S_N$ (resp.  $T_1, \dots, T_N).$
\begin{lemma}Keep the above notation.  We have \hspace{6cm}
\begin{enumerate}
\renewcommand{\theenumi}{\roman{enumi}}
\renewcommand{\labelenumi}{\textup{(\theenumi)}}
\item
\begin{equation*}
 \sum_{j=1}^N S_j S_j^* =P_A, 
\qquad
S_j^* S_j = \sum_{k=1}^N A(j,k)S_k S_k^* 
\quad\text{ for }\quad j=1,\dots,N,
\end{equation*}
and the $C^*$-algebra  
$
P_A \OAf P_A 
$
coincides with 
$C^*(S_1, \dots, S_N)$ that is  isomorphic to $\OA$.
\item
\begin{equation*}
 \sum_{j=1}^N T_j T_j^*  = P_{A^t}, 
\qquad
T_j^* T_j = \sum_{k=1}^N A^t(j,k)T_k T_k^*
\quad\text{ for }\quad j=1,\dots,N,
\end{equation*}
and the $C^*$-algebra  
$
P_{A^t}\OTAf P_{A^t}
$
coincides with
$C^*(T_1, \dots, T_N)$ that is isomorphic to $\OTA.$
\end{enumerate}
\end{lemma}
\begin{proof}
We will prove (i). 
%The other statement (ii) is shown in a similar way. 
By \eqref{eq:4.S2},
we have the following equalities
\begin{align*}
S_j^* S_j 
=& ( \tilde{S}_{j_0}\tilde{S}_{j_1}\tilde{S}_{j_2}\cdots \tilde{S}_{j_{m_j}})^*
    ( \tilde{S}_{j_0}\tilde{S}_{j_1}\tilde{S}_{j_2}\cdots \tilde{S}_{j_{m_j}}) \\
=& \tilde{S}_{j_{m_j}}^* \tilde{S}_{j_{m_j-1}}^* \cdots \tilde{S}_{j_1}^* \tilde{S}_{j_0}^*
     \tilde{S}_{j_0}\tilde{S}_{j_1}\cdots \tilde{S}_{j_{m_j-1}}\tilde{S}_{j_{m_j}} \\
=& \tilde{S}_{j_{m_j}}^* \tilde{S}_{j_{m_j-1}}^* \cdots \tilde{S}_{j_1}^* \tilde{S}_{j_1}
     \tilde{S}_{j_1}^*\tilde{S}_{j_1}\cdots \tilde{S}_{j_{m_j-1}}\tilde{S}_{j_{m_j}} \\
=& \tilde{S}_{j_{m_j}}^* \tilde{S}_{j_{m_j-1}}^* \cdots \tilde{S}_{j_2}^* \tilde{S}_{j_1}^*
     \tilde{S}_{j_1}\tilde{S}_{j_2}\cdots \tilde{S}_{j_{m_j-1}}\tilde{S}_{j_{m_j}} \\
=& \tilde{S}_{j_{m_j}}^* \tilde{S}_{j_{m_j-1}}^* \cdots \tilde{S}_{j_2}^* 
     \tilde{S}_{j_2}\cdots \tilde{S}_{j_{m_j-1}}\tilde{S}_{j_{m_j}}. 
\end{align*}
By continuing this procedure, the last term above goes to 
$\tilde{S}_{j_{m_j}}^*\tilde{S}_{j_{m_j}}$ 
so that 
$S_j^* S_j  =\tilde{S}_{j_{m_j}}^*\tilde{S}_{j_{m_j}}.$
We also have 
\begin{align*}
 S_j S_j^* 
=&( \tilde{S}_{j_0}\tilde{S}_{j_1}\tilde{S}_{j_2}\cdots \tilde{S}_{j_{m_j}})
    ( \tilde{S}_{j_0}\tilde{S}_{j_1}\tilde{S}_{j_2}\cdots \tilde{S}_{j_{m_j}})^*
     \\
=& \tilde{S}_{j_0}\tilde{S}_{j_1}\cdots \tilde{S}_{j_{m_j-1}}\tilde{S}_{j_{m_j}}
     \tilde{S}_{j_{m_j}}^* \tilde{S}_{j_{m_j-1}}^* \cdots \tilde{S}_{j_1}^* \tilde{S}_{j_0}^*
      \\
=& \tilde{S}_{j_0}\tilde{S}_{j_1}\cdots \tilde{S}_{j_{m_j-1}}\tilde{S}_{j_{m_j -1}}^*
     \tilde{S}_{j_{m_j-1}} \tilde{S}_{j_{m_j-1}}^* \cdots \tilde{S}_{j_1}^* \tilde{S}_{j_0}^*
      \\
=& \tilde{S}_{j_0}\tilde{S}_{j_1}\cdots \tilde{S}_{j_{m_j-1}}
     \tilde{S}_{j_{m_j-1}}^* \cdots \tilde{S}_{j_1}^* \tilde{S}_{j_0}^*.
\end{align*}
By continuing this procedure, the last term above goes to
$\tilde{S}_{j_0} \tilde{S}_{j_0}^*.$
%so that  
%$ S_j S_j^* =\tilde{S}_{j_0} \tilde{S}_{j_0}^*.$
Hence we have
\begin{equation}
S_j^* S_j = \tilde{S}_{j_{m_j}}^*\tilde{S}_{j_{m_j}}, \qquad
 S_j S_j^* = \tilde{S}_{j_0} \tilde{S}_{j_0}^*, \qquad j=1,\dots,N \label{eq:4.S4}
\end{equation}
so that 
\begin{equation}
\sum_{j=1}^N S_j S_j^* = P_A \quad\text{ and }\quad
S_j^* S_j = 
\sum_{k=1}^N A(j,k) S_k S_k^*,  \qquad j=1,\dots,N. \label{eq:4.S5}
\end{equation}
Similarly we have
 \begin{equation}
T_j^* T_j =\tilde{T}_{j_0}^*\tilde{T}_{j_0} , \qquad
T_j T_j^* = \tilde{T}_{j_{m_j}}\tilde{T}_{j_{m_j}}^*, \qquad j=1,\dots,N \label{eq:4.T4}
\end{equation}
so that 
\begin{equation}
 \sum_{j=1}^N T_j T_j^*=P_{A^t}  \quad\text{ and }\quad
T_j^* T_j = 
\sum_{k=1}^N A^t(j,k) T_k T_k^*,  \qquad j=1,\dots,N. \label{eq:4.T5}
\end{equation}
As 
$S_j = P_A \tilde{S}_{j_0}\tilde{S}_{j_1}\tilde{S}_{j_2}\cdots \tilde{S}_{j_{m_j}} P_A,$
one sees that  
$S_j \in P_A\OAf P_A$ for $j=1,\dots,N$. 
Hence we have
$C^*(S_1,\dots, S_N) \subset P_A\OAf P_A.$
We will show the converse inclusion relation.
We note that for $j_k, {j'}_{k'} \in \tilde{\Sigma},$
the equality 
$A_f(j_k, {j'}_{k'}) =1$
holds
 if and only if either of the following two cases occurs

(1)  $j'=j$ and $ k' = k+1$

(2)  $A(j,j') =1$ and $k=m_j, \, k' =0.$

\noindent
For 
$
\tilde{\mu} = (\tilde{\mu} _1,\dots, \tilde{\mu}_m), \, 
\tilde{\nu} = (\tilde{\nu} _1,\dots, \tilde{\nu}_n)
\in B_*(\bar{X}_{A_f}),
$
suppose that 
$P_A \tilde{S}_{\tilde{\mu}}\tilde{S}_{\tilde{\nu}}^* P_A \ne 0.$
We first see that 
$\tilde{\mu}_1 = m(1)_0$
and 
$\tilde{\nu}_1 = n(1)_0$
for some $m(1), n(1)= 1,\dots,N.$ 
By the conditions (1), (2), we know that 
%$\tilde{\mu} = (m(1)_0,\dots, i_{m_{i^{1}}}, \tilde{\mu} _1,\dots, \tilde{\mu}_m)
$$
\tilde{S}_{\tilde{\mu}} = S_{m(1)}\cdots S_{m(p)} \tilde{S}_{j_0} \cdots\tilde{S}_{j_k},
\quad
\tilde{S}_{\tilde{\nu}} = S_{n(1)}\cdots S_{n(q)} \tilde{S}_{i_0} \cdots\tilde{S}_{i_l}
$$
for some
$m(1), \dots, m(p), n(1), \dots, n(q) , j,i \in \{1,\dots,N\}$
and
$0\le k\le m_j, \, 0\le l \le m_i.$
Hence we may assume that 
$\tilde{\mu} = (j_0, j_1,\dots, j_k), \, 
  \tilde{\nu} = (i_0, i_1, \dots, i_l)
$
with 
$0\le k\le m_j, \, 0\le l \le m_i,$
so that 
$$
\tilde{S}_{\tilde{\mu}}\tilde{S}_{\tilde{\nu}}^*
= \tilde{S}_{j_0}\tilde{S}_{j_1}\cdots \tilde{S}_{j_k}
   \tilde{S}_{i_l}^*\cdots \tilde{S}_{i_1}^*\tilde{S}_{i_0}^*\ne 0.
$$
Since 
\begin{equation*}
\tilde{S}_{j_k} =\tilde{S}_{j_k} \tilde{S}_{j_k}^* \tilde{S}_{j_k} 
=\tilde{S}_{j_k} \tilde{S}_{j_{k+1}}\tilde{S}_{j_{k+1}}^*, \qquad 
\tilde{S}_{i_l}^* =\tilde{S}_{i_l}^* \tilde{S}_{i_l} \tilde{S}_{i_l}^* 
=\tilde{S}_{i_{l+1}} \tilde{S}_{i_{l+1}}^*\tilde{S}_{i_l}^*, \qquad 
\end{equation*}
we have
\begin{equation*}
\tilde{S}_{\tilde{\mu}}\tilde{S}_{\tilde{\nu}}^*
= \tilde{S}_{j_0}\tilde{S}_{j_1}\cdots \tilde{S}_{j_{k-1}}
\tilde{S}_{j_k} \tilde{S}_{j_{k+1}}\tilde{S}_{j_{k+1}}^*
 \tilde{S}_{i_{l+1}} \tilde{S}_{i_{l+1}}^*\tilde{S}_{i_l}^*
  \tilde{S}_{i_{l-1}}^*\cdots \tilde{S}_{i_1}^*\tilde{S}_{i_0}^*.
\end{equation*}
The condition
$
\tilde{S}_{\tilde{\mu}}\tilde{S}_{\tilde{\nu}}^*\ne 0
$
leads  $j_{k+1} = i_{l+1},$
so that 
we have $j_k = i_l, \, \cdots, j_1 = i_1, \, j_0 = i_0.$ 
Hence
\begin{align*}
& \tilde{S}_{j_0}\tilde{S}_{j_1}\cdots \tilde{S}_{j_{k-1}}
\tilde{S}_{j_k} \tilde{S}_{j_{k+1}}\tilde{S}_{j_{k+1}}^*
 \tilde{S}_{i_{l+1}} \tilde{S}_{i_{l+1}}^*\tilde{S}_{i_l}^*
  \tilde{S}_{i_{l-1}}^*\cdots \tilde{S}_{i_1}^*\tilde{S}_{i_0}^* \\
=& \tilde{S}_{j_0}\tilde{S}_{j_1}\cdots \tilde{S}_{j_{k-1}}
\tilde{S}_{j_k} \tilde{S}_{j_{k+1}}\tilde{S}_{j_{k+1}}^*
 \tilde{S}_{j_{k+1}} \tilde{S}_{j_{k+1}}^*\tilde{S}_{j_k}^*
  \tilde{S}_{j_{k-1}}^*\cdots \tilde{S}_{j_1}^*\tilde{S}_{j_0}^*.
\end{align*}
As 
$\tilde{S}_{j_{k+1}}^* \tilde{S}_{j_{k+1}}
=\tilde{S}_{j_{k+2}}\tilde{S}_{j_{k+2}}^*,
$
by continuing this procedure we know that 
\begin{equation*}
\tilde{S}_{\tilde{\mu}}\tilde{S}_{\tilde{\nu}}^*
= \tilde{S}_{j_0}\tilde{S}_{j_1}\cdots \tilde{S}_{j_{k-1}}
\tilde{S}_{j_k} \cdots \tilde{S}_{j_{m_j}}
\tilde{S}_{j_{m_j}}^*\cdots\tilde{S}_{j_k}^*
  \tilde{S}_{j_{k-1}}^*\cdots \tilde{S}_{j_1}^*\tilde{S}_{j_0}^*
= S_j S_j^*.
\end{equation*}  
This shows that 
the element
$%\tilde{S}_{\tilde{\mu}}\tilde{S}_{\tilde{\nu}}^*= 
P_A\tilde{S}_{\tilde{\mu}}\tilde{S}_{\tilde{\nu}}^*P_A$
belongs to the algebra
$C^*(S_1,\dots, S_N)$
 so that we obtain that
$P_A \OAf P_A =C^*(S_1,\dots, S_N)$
and similarly 
$P_{A^t} \OTAf P_{A^t} =C^*(T_1,\dots, T_N).$
\end{proof}

 Keep the above situation.
Recall that the vertex set 
$\V_f 
= \cup_{j=1}^N \{j_0, j_1, j_2, \dots,j_{m_j}\}$
 of the graph $\G_f$ is denoted by 
$\tilde{\Sigma}.$
We set
\begin{equation*}
\tilde{U}_{j_k} = \tilde{T}_{j_k}^*\otimes \tilde{S}_{j_k}, \qquad j_k \in \tilde{\Sigma}. 
\end{equation*}
The partial isometries $\tilde{U}_{j_k}, j_k \in \tilde{\Sigma}$ belong to the
Ruelle algebra $\mathcal{R}_{A_f}$ for the matrix $A_f$.
The following lemma is direct from the identities \eqref{eq:4.S2}, \eqref{eq:4.T2}.
\begin{lemma}\label{lem:6.2} For $j=1,\dots, N,\, k=0,1,\dots,m_j,$
we have
\begin{align*}
%( \tilde{U}_j\tilde{U}_{j_1}\tilde{U}_{j_2}\cdots \tilde{U}_{j_{m_j}})
 %   ( \tilde{U}_j\tilde{U}_{j_1}\tilde{U}_{j_2}\cdots \tilde{U}_{j_{m_j}})^*
  %   =&\tilde{U}_j\tilde{U}_j^* =\tilde{T}_j^* \tilde{T}_j \otimes \tilde{S}_j \tilde{S}_j^*\\
( \tilde{U}_{j_k}\tilde{U}_{j_{k+1}}\cdots \tilde{U}_{j_{m_j}})
    ( \tilde{U}_{j_k}\tilde{U}_{j_{k+1}}\cdots \tilde{U}_{j_{m_j}})^*
     =&\tilde{U}_{j_k}\tilde{U}_{j_k}^* 
=\tilde{T}_{j_k}^* \tilde{T}_{j_k} \otimes \tilde{S}_{j_k} \tilde{S}_{j_k}^*\\
%( \tilde{U}_j\tilde{U}_{j_1}\tilde{U}_{j_2}\cdots \tilde{U}_{j_{m_j}})^*
%    ( \tilde{U}_j\tilde{U}_{j_1}\tilde{U}_{j_2}\cdots \tilde{U}_{j_{m_j}})
 %    =&\tilde{U}_{j_{m_j}}^*\tilde{U}_{j_{m_j}} =T_j T_j^* \otimes S_j^* S_j\\
( \tilde{U}_{j_k}\tilde{U}_{j_{k+1}}\cdots \tilde{U}_{j_{m_j}})^*
    ( \tilde{U}_{j_{k}}\tilde{U}_{j_{k+1}}\cdots \tilde{U}_{j_{m_j}})
     =&\tilde{U}_{j_{m_j}}^*\tilde{U}_{j_{m_j}} =T_j T_j^* \otimes S_j^* S_j. 
\end{align*}
\end{lemma}
Hence we have
\begin{align*}
E_{A_f} 
=& \sum_{j=1}^N \sum_{k=0}^{m_j} \tilde{T}_{j_k}\tilde{T}_{j_k}^* \otimes
\tilde{S}_{j_k}^*\tilde{S}_{j_k}
= \sum_{j=1}^N \sum_{k=0}^{m_j} \tilde{U}_{j_k}\tilde{U}_{j_k}^*
\quad \text{ in } \quad\mathcal{R}_{A_f}, \\% \OTAf \otimes \OAf
E_A 
=& \sum_{j=1}^N  T_j T_j^*\otimes  S_j^* S_j
= \sum_{j=1}^N \tilde{U}_{j_{m_j}}^*\tilde{U}_{j_{m_j}}
\quad \text{ in } \quad \RA.
\end{align*}

%%%%%%%%%%%%%%%%%%%%%%%%%%%%%%%

Let $H$ be the separable infinite dimensional Hilbert space $\ell^2(\N)$. 
Take isometries $s_{j_k}, j_k \in \tilde{\Sigma}$ on $H$
 such that 
\begin{equation}
\sum_{k=0}^{m_j} s_{j_k}s_{j_k}^* = 1_H, \qquad  j=1,\dots,N.  \label{eq:6.11}
\end{equation}
%%%%%%%%%%%%%%%%%%%%%%%%%%%%%%
We define a partial isometry $\tilde{V}_f$ in the tensor product $C^*$-algebra 
$\mathcal{R}_{A_f} \otimes B(H)$ by setting 
\begin{equation}
\tilde{V}_f =
\sum_{j=1}^N\sum_{k=0}^{m_j} 
\tilde{U}_{j_k}\tilde{U}_{j_{k+1}}\cdots \tilde{U}_{j_{m_j}}
       \otimes s_{j_k}^*.    \label{eq:6.12}
\end{equation}
\begin{lemma}
$\tilde{V}_f \tilde{V}_f^* = E_{A_f} \otimes 1_H$ and 
$\tilde{V}_f^* \tilde{V}_f = E_A \otimes 1_H.$
\end{lemma}
\begin{proof}
By Lemma \ref{lem:6.2}
we have the following equalities.
\begin{align*}
  \tilde{V}_f^*\tilde{V}_f  
=&
\sum_{j=1}^N\sum_{k=0}^{m_j} (\tilde{U}_{j_k}\tilde{U}_{j_{k+1}}\cdots \tilde{U}_{j_{m_j}})^*
(\tilde{U}_{j_k}\tilde{U}_{j_{k+1}}\cdots \tilde{U}_{j_{m_j}})
       \otimes s_{j_k}s_{j_k}^* \\ 
=&
\sum_{j=1}^N
(T_j T_j^* \otimes S_j^* S_j)    
\otimes (\sum_{k=0}^{m_j}s_{j_k}s_{j_k}^*) \\
=&
\sum_{j=1}^N
T_j T_j^* \otimes S_j^* S_j    
\otimes 1_H \\
= & E_A \otimes 1_H
\end{align*}
and
\begin{align*}
  \tilde{V}_f \tilde{V}_f^*  
=&
\sum_{j=1}^N\sum_{k=0}^{m_j} 
(\tilde{U}_{j_k}\tilde{U}_{j_{k+1}}\cdots \tilde{U}_{j_{m_j}})
(\tilde{U}_{j_k}\tilde{U}_{j_{k+1}}\cdots \tilde{U}_{j_{m_j}})^*
       \otimes s_{j_k}^*s_{j_k} \\ 
=&
\sum_{j=1}^N
\sum_{k=0}^{m_j} \tilde{U}_{j_k} \tilde{U}_{j_k}^*  \otimes 1_H \\
= &  E_{A_f}\otimes 1_H.
\end{align*}
\end{proof}
Recall that $\K$ denotes  the $C^*$-algebra $\K(H)$ 
of compact operators on the Hilbert space $H.$ 
As 
$\tilde{U}_{j_k}$ belongs to $\mathcal{R}_{A_f}$
and
$ \mathcal{R}_{A_f}\otimes B(H)$
is contained in the multiplier algebra
$M(\mathcal{R}_{A_f}\otimes \K)$
of $\mathcal{R}_{A_f}\otimes \K,$
the partial isometry 
$\tilde{V}_f $ belongs to 
$M(\mathcal{R}_{A_f}\otimes \K)$.

%%%%%%%%%%%%%%%%%%%%%%%%%%
%We note that $\gamma^{A,f}_{(r,s)} (E_A) = E_A$, so that 
%$\gamma^{A,f}_{(r,s)}$ gives rise to an action of $\T^2$ on $\RA$.

\begin{lemma}\label{lem:6.4}
$E_{A_f}(P_{A^t}\otimes P_A) =(P_{A^t}\otimes P_A) E_{A_f}
=E_A.$ 
\end{lemma}
\begin{proof}
We have
\begin{align*}
 E_{A_f}(P_{A^t}\otimes P_A) 
 =&
 \sum_{j=1}^N (\sum_{k=0}^{m_j} \tilde{T}_{j_k}\tilde{T}_{j_k}^*\otimes\tilde{S}_{j_k}^*\tilde{S}_{j_k} )
\cdot( \sum_{j=1}^N \tilde{T}_{j_{m_j}}\tilde{T}_{j_{m_j}}^*\otimes
\sum_{i=1}^N\tilde{S}_{i_0}\tilde{S}_{i_0}^*) \\
=&
 \sum_{j=1}^N \{ \tilde{T}_{j_{m_j}}\tilde{T}_{j_{m_j}}^*\otimes
(\sum_{i=1}^N\tilde{S}_{j_{m_j}}^*\tilde{S}_{j_{m_j}}\cdot \tilde{S}_{i_0}\tilde{S}_{i_0}^*)\}
\end{align*}
By the identities \eqref{eq:4.S4}, \eqref{eq:4.T4}, 
we know 
$ 
\tilde{S}_{j_{m_j}}^*\tilde{S}_{j_{m_j}}=S_j^* S_j, \, 
\tilde{S}_{i_0}\tilde{S}_{i_0}^*=S_i S_i^*, \,
\tilde{T}_{j_{m_j}}\tilde{T}_{j_{m_j}}^*= T_j T_j^*,$
so that we have
\begin{equation*}
 E_{A_f}(P_{A^t}\otimes P_A) 
=
 \sum_{j=1}^N \{ T_j T_j^*\otimes
S_j^* S_j(\sum_{i=1}^N A(j,i)S_i S_i^*)\} 
=
 \sum_{j=1}^N 
( T_j T_j^*\otimes S_j^* S_j)
= E_A, 
\end{equation*}
and hence 
$(P_{A^t}\otimes P_A) E_{A_f}
=E_A.$ 
\end{proof}
\begin{theorem}\label{thm:6.5}
Let $A$ be an irreducible, non-permutation matrix
with entries in $\{0,1\}.$
For a continuous function $f:\bar{X}_A\longrightarrow \N$, 
there exists an isomorphism
$\Phi_f: \WRAf\otimes\K\longrightarrow \WRA\otimes\K$
of $C^*$-algebras such that 
\begin{equation*}
%\Phi_f(C(\bar{X}_{A_f})\otimes\C) =C(\bar{X}_A)\otimes\C), \qquad
\Phi_f\circ(\gamma^{A_f}_{(r,s)}\otimes\id) =
(\gamma^{A,f}_{(r,s)}\otimes\id)\circ\Phi_f, \quad (r,s) \in \T^2.
\end{equation*}
\end{theorem}
%We write this situation as
%\begin{equation*}
%(\WRAf\otimes\K, C(\bar{X}_{A_f})\otimes\C, \gamma^{A_f}\otimes\id, \T^2)\cong
%(\WRA\otimes\K, C(\bar{X}_A)\otimes\C, \gamma^{A,f}\otimes\id, \T^2)
%\end{equation*}
\begin{proof}
By Corollary \ref{cor:5.5},
we may assume that $f$ is of the form
$\sum_{j=1}^N f_j \chi_{U_j(0)}.$
Define
$\Phi_f: \WRAf\otimes\K\longrightarrow \WRA\otimes\K$
by setting
$\Phi_f(x\otimes K) = \tilde{V}_f^*(x\otimes K)\tilde{V}_f
$ for 
$x\otimes K\in \WRAf\otimes\K.$
We note that 
\begin{gather*}
\WRAf =E_{A_f}(\mathcal{O}_{(A_f)^t}\otimes \mathcal{O}_{A_f})E_{A_f},
\qquad
\WRA = E_A(\OTA\otimes\OA)E_A, \\
\tilde{V}_f\tilde{V}_f^* = E_{A_f}\otimes 1_H, \qquad
\tilde{V}_f^*\tilde{V}_f = E_A\otimes 1_H, \\
(P_{A^t} \otimes P_A)(\mathcal{O}_{(A_f)^t}\otimes \mathcal{O}_{A_f})
(P_{A^t} \otimes P_A) = \OTA\otimes\OA.
\end{gather*}
For 
$x\otimes K\in \WRAf\otimes\K,$
we have
\begin{align*}
& \tilde{V}_f^*(x\otimes K)\tilde{V}_f \\
=& (E_A\otimes 1_H)\tilde{V}_f^*(x\otimes K)\tilde{V}_f (E_A\otimes 1_H) \\
=& (E_A\otimes 1_H)((P_{A^t} \otimes P_A)\otimes1_H)
\tilde{V}_f^*(x\otimes K)\tilde{V}_f ((P_{A^t} \otimes P_A)\otimes1_H) (E_A\otimes 1_H). 
\end{align*}
As $\tilde{V}_f \in M(\WRAf\otimes\K)$,
the element
$\tilde{V}_f^*(x\otimes K)\tilde{V}_f$ belongs to $\WRAf\otimes\K$
which is 
$(E_{A_f}\otimes 1_H)((\mathcal{O}_{(A_f)^t}\otimes \mathcal{O}_{A_f})\otimes\K)
(E_{A_f}\otimes 1_H).$
By Lemma \ref{lem:6.4}, we have
\begin{equation*}
(E_A\otimes 1_H)((P_{A^t} \otimes P_A)\otimes1_H)(E_{A_f}\otimes 1_H)
=(E_A\otimes1_H)((P_{A^t} \otimes P_A)\otimes1_H)
\end{equation*}
so that 
$\tilde{V}_f^*(x\otimes K)\tilde{V}_f$
belongs to the algebra
$$
(E_A\otimes1_H)((P_{A^t} \otimes P_A)\otimes1_H)
((\mathcal{O}_{(A_f)^t}\otimes \mathcal{O}_{A_f})\otimes\K)
((P_{A^t} \otimes P_A)\otimes1_H)
(E_A\otimes1_H)
$$
which is 
$$
%(E_A\otimes1_H)((P_{A^t} \otimes P_A)\otimes1_H)
%((\mathcal{O}_{(A_f)}^t\otimes \mathcal{O}_{A_f})\otimes\K)
%((P_{A^t} \otimes P_A)\otimes1_H)
%(E_A\otimes1_H)=
E_A(\OTA\otimes\OA)E_A\otimes\K = \WRA\otimes\K.
$$
Hence 
$\tilde{V}_f^*(x\otimes K)\tilde{V}_f$ 
belongs to 
the algebra 
$\WRA\otimes\K.$
This shows that the inclusion relation
\begin{equation}
\tilde{V}_f^*(\WRAf\otimes K)\tilde{V}_f 
\subset \WRA\otimes\K \label{eq:6.13}
\end{equation}
holds.
Conversely, 
for 
$y\otimes K\in \WRA\otimes\K,$
we have
\begin{equation*}
 \tilde{V}_f(y\otimes K)\tilde{V}_f^* 
= (E_{A_f}\otimes 1_H)\tilde{V}_f(y\otimes K)\tilde{V}_f^* (E_{A_f}\otimes 1_H). 
\end{equation*}
As $\tilde{V}_f \in M(\WRAf\otimes\K)$
and
$y\otimes K\in \WRA\otimes\K\subset \WRAf\otimes\K,$
the element
$\tilde{V}_f(y\otimes K)\tilde{V}_f^*$ 
belongs to $\WRAf\otimes\K$
which is 
$(E_{A_f}\otimes 1_H)((\mathcal{O}_{(A_f)^t}\otimes \mathcal{O}_{A_f})\otimes\K)
(E_{A_f}\otimes 1_H).$
Hence 
$\tilde{V}_f(x\otimes K)\tilde{V}_f^*$
belongs to the algebra
$$
(E_{A_f}\otimes1_H)(
E_{A_f}\otimes 1_H)((\mathcal{O}_{(A_f)^t}\otimes \mathcal{O}_{A_f})\otimes\K)
(E_{A_f}\otimes 1_H)
$$
which is 
$$
E_{A_f}(\mathcal{O}_{(A_f)^t}\otimes \mathcal{O}_{A_f})E_{A_f}\otimes\K
=\WRAf\otimes\K.
$$
Hence 
$\tilde{V}_f(y\otimes K)\tilde{V}_f^*$ 
belongs to 
the algebra 
$\WRAf\otimes\K.$
This shows that the inclusion relation
\begin{equation}
\tilde{V}_f(\WRA\otimes K)\tilde{V}_f^* 
\subset \WRAf\otimes\K \label{eq:6.14}
\end{equation}
holds.
Since $\tilde{V}_f^*\tilde{V}_f = E_A\otimes1_H$
and
$E_A\WRA E_A =\WRA$, the inclusion relation
\eqref{eq:6.14} implies
\begin{equation}
\WRA\otimes K 
\subset \tilde{V}_f^*(\WRAf\otimes\K)\tilde{V}_f.
\label{eq:6.15}
\end{equation}
By \eqref{eq:6.13} and \eqref{eq:6.15}, 
we have
$
\tilde{V}_f^*(\WRAf\otimes\K)\tilde{V}_f=\WRA\otimes\K.
$
Therefore we have an isomorphism
$\Phi_f = \Ad(\tilde{V}_f^*): \WRAf\otimes\K \longrightarrow \WRA\otimes\K.
$

Since 
$\gamma^{A_f}_{(r,s)}(\tilde{U}_{j_k}) = \tilde{U}_{j_k}$
for $j_k \in \tilde{\Sigma}$ by Lemma \ref{lem:5.3},
we know 
$(\gamma^{A_f}_{(r,s)}\otimes\id)(\tilde{V}_f) = \tilde{V}_f.$
For $x\otimes K\in \WRAf\otimes\K,$ 
we have 
\begin{align*}
(\Phi_f\circ(\gamma^{A_f}_{(r,s)}\otimes\id))(x\otimes K)
= & \tilde{V}_f^*(\gamma^{A_f}_{(r,s)}(x)\otimes K) \tilde{V}_f \\
= & (\gamma^{A_f}_{(r,s)}\otimes\id)(\tilde{V}_f^*(x\otimes K) \tilde{V}_f) \\
= & (\gamma^{A_f}_{(r,s)}\otimes\id)(\Phi(x\otimes K)).
\end{align*}
Now
\begin{align*}
    \gamma^{A_f}_{(r,s)}(T_j^*\otimes S_k) 
 = &\gamma^{A_f}_{(r,s)}( 
(\tilde{T}_{j_{m_j}}\cdots \tilde{T}_{j_1} \tilde{T}_{j_0})^*\otimes
\tilde{S}_{k_0} \tilde{S}_{k_1}\cdots\tilde{S}_{k_{m_k}}) \\
 = &
\exp{(2\pi\sqrt{-1} (f_k s - f_j r))}(
(\tilde{T}_{j_{m_j}}\cdots \tilde{T}_{j_1} \tilde{T}_{j_0})^*\otimes
(\tilde{S}_{k_0} \tilde{S}_{k_1}\cdots\tilde{S}_{k_{m_k}}) ) \\
= &\alpha^{A^t,f}_r(T_j^*) \otimes \alpha^{A,f}_s(S_k).
\end{align*} 
Hence the restriction of $\gamma^{A_f}_{(r,s)}\otimes\id$
to the subalgebra $\WRA\otimes\K$ coincides with 
 $\gamma^{A,f}_{(r,s)}\otimes\id$
so that we conclude that 
$$\Phi_f\circ (\gamma^{A_f}_{(r,s)}\otimes\id)
=  (\gamma^{A,f}_{(r,s)}\otimes\id)\circ\Phi_f.
$$
\end{proof}
By using Proposition \ref{prop:4.5}, 
 we have the converse implication of Theorem \ref{thm:6.5}  in the following way.
\begin{theorem}\label{thm:6.6}
Let $A, B$ be irreducible, non-permutation  matrices with entries in $\{0,1\}.$
Suppose that there exist a continuous function   
$f :\bar{X}_A \longrightarrow \N$ and an isomorphism
$\Phi:\WRB\otimes\K\longrightarrow \WRA\otimes\K$ such that 
\begin{equation*}
%\Phi(C(\bar{X}_B\otimes\C))=\bar{X}_A\otimes\C,\qquad
\Phi\circ (\gamma^B_{(r,s)}\otimes\id)
=  (\gamma^{A, f}_{(r,s)}\otimes\id)\circ\Phi, \qquad (r,s) \in \T^2.
\end{equation*}
Then the two-sided topological Markov shifts 
$(\bar{X}_B,\bar{\sigma}_B)$ and 
$(\bar{X}_A,\bar{\sigma}_A)$ are flow equivalent.
\end{theorem}
\begin{proof}
By Theorem \ref{thm:6.5}, 
there exists an isomorphism
$\Phi_f: \WRAf\otimes\K\longrightarrow \WRA\otimes\K$
of $C^*$-algebras such that 
\begin{equation*}
%\Phi_f(C(\bar{X}_{A_f})\otimes\C) =C(\bar{X}_A)\otimes\C),\qquad
\Phi_f\circ(\gamma^{A_f}_{(r,s)}\otimes\id) =
(\gamma^{A,f}_{(r,s)}\otimes\id)\circ\Phi_f, \quad (r,s) \in \T^2.
\end{equation*}
We define the isomorphism
$\Phi_\circ=\Phi_f^{-1}\circ\Phi:
\WRB\otimes\K\longrightarrow \WRAf\otimes\K$
which satisfies
\begin{equation*}
%\Phi_\circ(C(\bar{X}_B)\otimes\C) =C(\bar{X}_{A_f})\otimes\C),\qquad
\Phi_\circ\circ(\gamma^B_{(r,s)}\otimes\id) =
(\gamma^{A_f}_{(r,s)}\otimes\id)\circ\Phi_\circ, \quad (r,s) \in \T^2.
\end{equation*}
By Proposition \ref{prop:4.5}, 
we know that 
$(\bar{X}_B,\bar{\sigma}_B)$ and 
$(\bar{X}_{A_f},\bar{\sigma}_{A_f})$ are flow equivalent.
Since
$(\bar{X}_{A_f},\bar{\sigma}_{A_f})$
is a discrete suspension of
$(\bar{X}_A,\bar{\sigma}_A)$,
they are flow equivalent, so that  
$(\bar{X}_B,\bar{\sigma}_B)$ and 
$(\bar{X}_A,\bar{\sigma}_A)$ are flow equivalent.
\end{proof}
Therefore we have a characterization of flow equivalence in terms of 
the $C^*$-algebras $\WRA$ with their gauge actions with potentials. 
\begin{theorem}\label{thm:6.7}
Let $A, B$ be irreducible, non-permutation matrices with entries in $\{0,1\}.$
Two-sided topological Markov shifts 
$(\bar{X}_B,\bar{\sigma}_B)$ and 
$(\bar{X}_A,\bar{\sigma}_A)$ are flow equivalent
if and only if 
there exist an irreducible non-permutation matrix  $C$ with entries in $\{0,1\}$
and continuous functions
$f_A, f_B:\bar{X}_C \longrightarrow \N$
with its values in positive integers
such that 
there exist isomorphisms
\begin{gather}
\Phi_A: \WRA\otimes\K\longrightarrow \WRC\otimes\K
\quad\text{ satisfying }\quad
\Phi_A\circ(\gamma^A_{(r,s)}\otimes\id) =
(\gamma^{C,f_A}_{(r,s)}\otimes\id)\circ\Phi_A \label{eq:6.AC} \\
\intertext{and} 
\Phi_B: \WRB\otimes\K\longrightarrow \WRC\otimes\K
\quad\text{ satisfying }\quad
\Phi_B\circ(\gamma^B_{(r,s)}\otimes\id) =
(\gamma^{C,f_B}_{(r,s)}\otimes\id)\circ\Phi_B. \label{eq:6.BC}
\end{gather}
\end{theorem}
\begin{proof}
Suppose that 
the two-sided topological Markov shifts 
$(\bar{X}_B,\bar{\sigma}_B)$ and 
$(\bar{X}_A,\bar{\sigma}_A)$ are flow equivalent.
By Parry--Sullivan \cite{PS}, 
there exist an irreducible non-permutation matrix $C$ with entries in $\{0,1\}$
and continuous functions
$f_A, f_B:\bar{X}_C \longrightarrow \N$
such that 
$(\bar{X}_A,\bar{\sigma}_A)$ and
$(\bar{X}_{C_{f_A}},\bar{\sigma}_{C_{f_A}})$ 
 are topologically conjugate,
and
$(\bar{X}_B,\bar{\sigma}_B)$ and
$(\bar{X}_{C_{f_B}},\bar{\sigma}_{C_{f_B}})$ 
 are topologically conjugate.
By \cite[Theorem 1.1]{MaPre2017b} and Theorem \ref{thm:6.5},
 we have isomorphisms
$\Phi_A: \WRA\otimes\K\longrightarrow \WRC\otimes\K$
and
$\Phi_B: \WRB\otimes\K\longrightarrow \WRC\otimes\K$
of $C^*$-algebras satisfying \eqref{eq:6.AC} and \eqref{eq:6.BC}, respectively.

The converse implication immediately follows from
Theorem \ref{thm:6.6}. 
\end{proof}
As a corollary we have
\begin{corollary}\label{cor:4.8}
Let $A, B$ be irreducible, non-permutation matrices with entries in $\{0,1\}.$
Two-sided topological Markov shifts 
$(\bar{X}_B,\bar{\sigma}_B)$ and 
$(\bar{X}_A,\bar{\sigma}_A)$ are flow equivalent
if and only if 
there exist continuous functions
$f_A:\bar{X}_A \longrightarrow \N$
and
$f_B:\bar{X}_B \longrightarrow \N$
such that 
there exists an isomorphism
$\Phi: \WRA\otimes\K \longrightarrow \WRB\otimes\K$ 
of $C^*$-algebras satisfying 
\begin{equation*}
%\Phi( C(\bar{X}_A)\otimes\C) = C(\bar{X}_B)\otimes\C, \qquad
\Phi\circ(\gamma^{A,f_A}_{(r,s)}\otimes\id)
=(\gamma^{B,f_B}_{(r,s)}\otimes\id)\circ\Phi,
\qquad
(r,s) \in \T^2.
\end{equation*}
\end{corollary}
\begin{proof}
By Parry--Sullivan \cite{PS}, 
the two-sided topological Markov shifts 
$(\bar{X}_B,\bar{\sigma}_B)$ and 
$(\bar{X}_A,\bar{\sigma}_A)$ are flow equivalent
if and only if 
there exist continuous functions
$f_A:\bar{X}_A \longrightarrow \N$
and
$f_B:\bar{X}_B \longrightarrow \N$
such that 
$(\bar{X}_{A_{f_A}},\bar{\sigma}_{A_{f_A}})$ 
 and
$(\bar{X}_{B_{f_B}},\bar{\sigma}_{B_{f_B}})$ 
 are topologically conjugate.
The assertion follows from \cite[Theorem 1.1]{MaPre2017b} and Theorem \ref{thm:6.5},
 Theorem \ref{thm:6.6}. 
\end{proof}

%%%%%%%%%%%%%%%%%%%%%%%%%%%%%%%%%%%%%%%%%%%
%%%%%%%%%%%%%%%%%%%%%%%%%%%%%%%%%%%%%%%%%%%%%%%%
\section{Detail statement of Theorem \ref{thm:6.5}}
%%%%%%%%%%%%%%%%%%%%%%%%%%%%%%%%%%%%%%%%%%%%%%%
%%%%%%%%%%%%%%%%%%%%%%%%%%%%%%%%%%%%%%%%%%%%%%%
In this final section, 
we describe  a detailed statement of
Theorem \ref{thm:6.5}.
We keep the assumption that  $A$ is an irreducible, non-permutation 
matrix with entries in $\{0,1\},$
and $f:\bar{X}_A \longrightarrow \N$ is a continuous function such that 
$f = \sum_{j=1}^N f_j \chi_{U_j(0)}$
for some positive integers $f_j.$
Let $e_i, i\in \N$ be a sequence of vectors of complete orthonormal basis of the separable infinite dimensional Hilbert space $H =\ell^2(\N).$
We fix $j =1,\dots,N.$
For $k=0,1,\dots,m_j,$ where $m_j = f_j -1,$
we set 
$$
\N_k = \{ n \in \N \mid n \equiv k \,  ({\operatorname{mod}}\,\, f_j) \}
$$
so that we have a disjoint union 
$\N = \N_0 \cup \N_1 \cup\cdots \cup \N_{m_j}.$
We write 
$$
\N_k = \{ 1_k, 2_k,3_k, \dots \}
\quad
\text{ where }
1_k < 2_k <3_k< \dots.
$$
Define an isometry $s_{j_k}$ on $H$ by setting
\begin{equation}
s_{j_k} e_n = e_{n_k}, \qquad n \in \N, \,\, k=0,1,\dots,m_j.
 \label{eq: 7.1}
\end{equation}
The family 
$\{s_{j_k} \}_{k=0}^{m_j}$ satisfy \eqref{eq:6.11}.
We may construct the operator 
$\tilde{V}_f$ from them by the formula \eqref{eq:6.12}.
As in the proof of Theorem \ref{thm:6.5},
we define the isomorphism
$\Phi_f: \WRAf\otimes\K\longrightarrow \WRA\otimes\K$
of $C^*$-algebras by 
$\Phi_f(x\otimes K) = \tilde{V}_f^*(x\otimes k) \tilde{V}_f$
for $x \otimes K \in \WRAf.$
As $E_{A_f}\otimes 1_H = \tilde{V}_f \tilde{V}_f^*,$
we have
\begin{align*}
  \tilde{V}_f^* (E_{A_f}\otimes e_1)\tilde{V}_f  
=&
\sum_{j=1}^N\sum_{k=0}^{m_j} (\tilde{U}_{j_k}\tilde{U}_{j_{k+1}}\cdots \tilde{U}_{j_{m_j}})^*
(\tilde{U}_{j_k}\tilde{U}_{j_{k+1}}\cdots \tilde{U}_{j_{m_j}})
       \otimes s_{j_k} e_1 s_{j_k}^* \\ 
=&
\sum_{j=1}^N \{
(T_j T_j^* \otimes S_j^* S_j)    
\otimes (\sum_{k=0}^{m_j}s_{j_k} e_1 s_{j_k}^*) \}.
\end{align*}
Since $s_{j_k} e_1 s_{j_k}^* = e_{1_k},$
we obtain the following formula in the $K_0$-group $K_0(\WRA)$
\begin{equation*}
[\tilde{V}_f^* (E_{A_f}\otimes e_1)\tilde{V}_f ] 
%=\sum_{j=1}^N [T_j T_j^* \otimes S_j^* S_j]    
%\otimes [\sum_{k=0}^{m_j}s_{j_k} e_1 s_{j_k}^*] \\
=
\sum_{j=1}^N m_j [T_j T_j^* \otimes S_j^* S_j]    
\quad \text{ in } K_0(\WRA).
 \end{equation*}
In the group $K_0(\WRA),$
we define 
\begin{equation*}
f([E_A]) := 
\sum_{j=1}^N m_j [T_j T_j^* \otimes S_j^* S_j]    
\quad \text{ in } K_0(\WRA).
\end{equation*}
Recall that $\C$ denotes the commutative $C^*$-algebra on $H$ consisting of diagonal operators with respect to the basis $\{e_n\}_{n \in \N}.$
The construction of the operator $\tilde{V}_f$ tells us that 
the equality
\begin{equation*}
\tilde{V}_f^* (C(\bar{X}_A)\otimes\C)\tilde{V}_f =C(\bar{X}_A)\otimes\C
\end{equation*}
holds. Therefore we obtain the following proposition
\begin{proposition}\label{prop:7.1}
Let $A$ be an irreducible, non-permutation matrix
with entries in $\{0,1\}.$
Then there exists an isomorphism
$\Phi_f: \WRAf\otimes\K\longrightarrow \WRA\otimes\K$
of $C^*$-algebras such that 
\begin{gather*}
\Phi_f(C(\bar{X}_{A_f})\otimes\C) =C(\bar{X}_A)\otimes\C, \\
\Phi_f\circ(\gamma^{A_f}_{(r,s)}\otimes\id) =
(\gamma^{A,f}_{(r,s)}\otimes\id)\circ\Phi_f, \quad (r,s) \in \T^2, \\
\Phi_{f*}([E_{A_f}]) = f([E_A]) \quad \text{ in } K_0(\WRA).
 \end{gather*}
%We write this situation as
%\begin{equation*}
%(\WRAf\otimes\K, C(\bar{X}_{A_f})\otimes\C, \gamma^{A_f}\otimes\id, \T^2)\cong
%(\WRA\otimes\K, C(\bar{X}_A)\otimes\C, \gamma^{A,f}\otimes\id, \T^2)
%\end{equation*}
\end{proposition}
%%%%%%%%%%%%%%%%%%%%%%%%%%%%%%

\bigskip

%%%%%%%%%%%%%%%%%%%%%%
{\it Acknowledgments:}
%The author would like to deeply thank 
%Ian F. Putnam for his comments and suggestions on the earlier version of the paper.
This work was supported by JSPS KAKENHI Grant Number 15K04896.

%%%%%%%%%%%%%%%%%%%%%%%%%%%%%%%%%%%%%%%%%%%%%%%%%
%%%%%%%%%%%%%%%%%%%%%%%%%%%%%%%%%%%%%%%%%%%%%%%%%%


\begin{thebibliography}{99}



%\bibitem{AnanRenault}
%{\sc C. Anantharaman-Delaroche and J. Renault},
%{\it Amenable Groupoids}, 
%L'Enseignement Math\'ematique Gen\'eve, 2000.

%\bibitem{Bowen1}{\sc R. Bowen},
%{\it Markov partitions for axiom A diffeomorphisms},
%Amer.\ J.\  Math.\  {\bf 92}(1970), pp.\ 725--747


\bibitem{Bowen2}{\sc R. Bowen},
{\it Equilibrium states and the ergodic theory of Anosov diffeomorphisms}
Lecture notes in Math. Springer, Berlin 1975, No. 475.

\bibitem{BF}{\sc R. Bowen and J. Franks},
{\it Homology for zero-dimensional nonwandering sets},
Ann.\  Math.\ {\bf 106}(1977), pp.\ 73--92.


\bibitem{BH}{\sc M. Boyle and D. Handelman},
{\it Orbit equivalence, flow equivalence and ordered cohomology},
Israel J.\  Math.\
{\bf 95}(1996), pp.\ 169--210.

\bibitem{Brown}
{\sc L. G. Brown},
{\it Stable isomorphism of hereditary subalgebras of $C^*$-algebras},
Pacific J.\ Math.\ {\bf 71}(1977), pp.\ 335--348.


\bibitem{CEOR}
{\sc T. M. Carlsen, S. Eilers, E. Ortega, and G. Restorff},
{\it Flow equivalence and orbit equivalence for shifts of finite type and isomorphism of their groupoids},
preprint, arXiv: 1610.09945.


\bibitem{CRS}
{\sc T. M. Carlsen, E. Ruiz and A. Sims},
{\it Equivalence and stable isomorphism of groupoids,
and diagonal-preserving stable isomorphisms of graph $C^*$-algebras 
and Leavitt path algebras},
Proc. Amer. Math. Soc. {\bf 145}(2017), pp. \ 1581--1592.



%\bibitem{CR}
%{\sc T. M. Carlsen and J. Rout},
%{\it Diagonal-preserving gauge invariant isomorphisms of graph $C^*$-algebras},
%preprint, arXiv: 1610.00692 [math. OA].


%\bibitem{CRS}
%{\sc T. M. Carlsen, E. Ruiz and A. Sims},
%{\it Equivalence and stable isomorphism of groupoids, and diagonal-preserving 
%stable isomorphisms of graph $C^*$-algebras and Leavitt path algebras},
%preprint, arXiv: 1602.02602 [math. OA].


\bibitem{CuntzInvent80}
{\sc J. ~Cuntz}, 
{\it A class of $C^*$-algebras and topological Markov chains II: reducible chains and the Ext-functor for $C^*$-algebras},
Invent.\ Math.\ {\bf 63}(1980), pp.\ 25--40.

%\bibitem{Cu4}
%{\sc J. ~Cuntz}, 
%{\it On the homotopy groups of the space of endomorphisms of 
%a $C^*$-algebra (with applications to topological Markov chains)},
% Operator algebras and group representations, vol. I (Neptun, 1980), 
%124--137, Monogr. Stud. Math., 17, Pitman, Boston, MA, 1984.


\bibitem{CK}{\sc J. ~Cuntz and W. ~Krieger},
{\it A class of $C^*$-algebras and topological Markov chains},
 Invent.\ Math.\  {\bf 56}(1980), pp.\ 251--268.


\bibitem{Effros}
{\sc E. Effros},
{\it Dimensions and $C^*$-algebras},
 CBMS Regional Conference Series in Math. {\bf 46}(1979). Amer. Math. Soc..

%\bibitem{EFW1981}
%{\sc M. Enomoto, M. Fujii and Y. Watatani},
%{\it $K_0$-groups and classifications of Cuntz--Krieger algebras}, 
%Math.\ Japon. {\bf 26}(1981), pp.\ 443--460.


%\bibitem{EFW1984}{\sc M. Enomoto, M. Fujii and Y. Watatani},
%{\it KMS states for gauge action on $\OA$}, 
%Math.\ Japon {\bf 29}(1984), pp.\ 607--619.

\bibitem{EvKishi}{\sc D. E. Evans and A. Kishimoto},
{\it Trace scaling automorphisms of certain stable AF algebras}, 
Hokkaido Math. J. {\bf 26}(1997), pp. \. 211--224.




\bibitem{Franks}{\sc J. Franks},
{\it Flow equivalence of subshifts of finite type},
Ergodic Theory Dynam. Systems {\bf 4}(1984), pp.\ 53--66.

\bibitem{Holton}{\sc C. G. Holton},
{\it The Rohlin property for shifts of finite type},
J. Funct. Anal. {\bf 229}(2005), pp. \ 277--299.


\bibitem{KamPut}{\sc J. Kaminker and I. F. Putnam},
{\it K-theoretic duality of shifts of finite type},
Comm. Math. Phys. {\bf 187}(1997), pp. \ 509--522.


\bibitem{KamPutSpiel}{\sc J. Kaminker, I. F. Putnam and J. Spielberg},
{\it Operator algebras and hyperbolic dynamics},
Operator algebras and quantum field theory (Rome, 1996), 
525--532, Int. Press, Cambridge, MA, 1997.


\bibitem{KilPut}
{\sc D. B. Killough and I. F. Putnam},
{\it Ring and module structures on dimension groups 
associated with a shift of finite type},
Ergodic Theory Dynam. Systems {\bf 32}(2012), pp.\ 1370--1399.
 

 \bibitem{Kirchberg}
{\sc E. Kirchberg},
{\it The classification of purely infinite $C^*$-algebras using Kasparov's theory}, 
 preprint, 1994.

\bibitem{Krieger1977}
{\sc W. Krieger},
{\it On topological Markov chains},
 Dynamical systems, Vol. II Warsaw, pp. \ 193--196. 
Ast\'erisque, No. 50, Soc. Math. France, Paris, 1977.

\bibitem{Krieger1979}
{\sc W. Krieger},
{\it On a dimension for a class of homeomorphism groups},
 Math. Ann. \ {\bf 252}(1979/80), pp.\ 87--95.


\bibitem{Krieger1980}
{\sc W. Krieger},
{\it On dimension functions and topological Markov chains},
 Invent. Math. \ {\bf 56}(1980), pp. \ 239--250.


\bibitem{LM}{\sc D. ~Lind and B. ~Marcus},
{\it An introduction to symbolic dynamics and coding},
 Cambridge University Press, Cambridge (1995).


\bibitem{MaMathScand1998}
{\sc K Matsumoto},
{\it K-theory for $C^*$-algebras associated with subshifts}, 
 Math. Scand.  {\bf 82}(1998), pp. \. 237--255.

%\bibitem{Ma2009} {\sc K. Matsumoto},
%{\it  Orbit equivalence in $C^*$-algebras defined by actions of  symbolic dynamical systems},
%Contemporary Math.  {\bf 503}(2009),  pp.\ 121--140.



%\bibitem{MaETDS2004} {\sc K. Matsumoto},
%{\it Strong shift equivalence of symbolic dynamical systems 
%and Morita equivalence of $C^*$-algebras},
%Ergodic Theory Dynam. Systems {\bf 24}(2004), pp.\ 199--215.



%\bibitem{MaPacific}{\sc K. Matsumoto},
%{\it Orbit equivalence of topological Markov shifts and Cuntz--Krieger algebras},
%Pacific J.\ Math.\  {\bf 246}(2010), 199--225.



%\bibitem{MaYMJ}
%{\sc K. Matsumoto},
%{\it Some remarks  on orbit equivalence of 
%topological Markov shifts and Cuntz--Krieger algebras},
%Yokohama Math. J.
%{\bf 58}(2012), pp.\ 41--52.


%\bibitem{Ma16}
%{\sc K. Matsumoto},
%{\it Orbit equivalence of one-sided subshifts and the associated $C^*$-algebras},
%preprint, math  arXiv:0709.1185.


%\bibitem{MaDCDS}
%{\sc K. Matsumoto},
%{\it K-groups of the full group actions  
%on one-sided topological Markov shifts},
% Discrete and Contin. Dyn. Syst.
%{\bf 33}(2013),  pp.\ 3753--3765.


%\bibitem{MaPAMS}{\sc K. Matsumoto},
%{\it  Classification of Cuntz--Krieger algebras by orbit equivalence of topological Markov shifts}, Proc. Amer. Math. Soc. {\bf 141}(2013), pp.\ 2329--2342.


%\bibitem{MaPre2012}{\sc K. Matsumoto},
%{\it Full groups of one-sided topological Markov shifts},
%preprint, arXiv:1205.1320, to appear in Israel J. Math..

%\bibitem{MaJOT2015}{\sc K. Matsumoto},
%{\it Strongly continuous orbit equivalence of 
%one-sided topological Markov shifts},
%J. Operator Theory {\bf 74}(2015), pp. 101--127.




%\bibitem{MaPre2015}{\sc K. Matsumoto},
%{\it Continuous orbit equivalence, flow equivalence of Markov shifts and circle actions on Cuntz--Krieger algebras},
%Math. Z. (2016), DOI:10.1007/s00209-016-1700-3, 21 pages.  



%\bibitem{MaPre20162} {\sc K. Matsumoto},
%{\it Strong shift equivalence of symbolic dynamical systems 
%and Morita equivalence of $C^*$-algebras II}, in preparation.



%\bibitem{MaJOT2015}{\sc K. Matsumoto},
%{\it Strongly continuous orbit equivalence of 
%one-sided topological Markov shifts},
%J. Operator Theory {\bf 74}(2015), pp. 101--127.





%\bibitem{MaProcAMS2017}
%{\sc K. Matsumoto},
%{\it Uniformly continuous orbit equivalence of Markov shifts and gauge actions on %Cuntz--Krieger algebras}, Proc. Amer. Math. Soc. {\bf 145}(2017), pp. \ 1131--1140.  



%\bibitem{MaMZ2017}
%{\sc K. Matsumoto},
%{\it Continuous orbit equivalence, flow equivalence of Markov shifts and circle actions on %Cuntz--Krieger algebras}, Math. Z. {\bf 285}(2017), pp. 121--141.  


%\bibitem{MaDocMath2017} {\sc K. Matsumoto},
%{\it Topological conjugacy of topological Markov shifts 
%and Cuntz--Krieger algebras}, 
%Documenta Math. {\bf 22}(2017), pp.\ 873--915.
%

\bibitem{MaPre2017a} {\sc K. Matsumoto},
{\it Asymptotic continuous orbit equivalence of Smale spaces and Ruelle algebras},
preprint, arXiv: 1703.07011, to appear in Canad. J. Math..


\bibitem{MaPre2017b} {\sc K. Matsumoto},
{\it Topological conjugacy of topological Markov shifts and Ruelle algebras},
preprint, arXiv: 1706.07155.



\bibitem{MMKyoto}{\sc K. Matsumoto and H. Matui},
{\it Continuous orbit equivalence of topological Markov shifts 
and Cuntz--Krieger algebras},
Kyoto J. Math. \ {\bf 54}(2014), pp.\ 863--878.

%\bibitem{MMEDTD} {\sc K. Matsumoto and H. Matui},
%{\it Continuous orbit equivalence of topological Markov shifts 
%and dynamical zeta functions},
%Ergodic Theory Dynam. Systems {\bf 36}(2016), pp. \ 1557--1581.


%\bibitem{Matui2006} {\sc H. Matui}, 
%{\it Some remarks on topological full groups of  Cantor minimal systems},
%Internat. J.  Math. {\bf 17}(2006),  pp.\ 231--251.

%\bibitem{MatuiPre2011}{\sc H. Matui}, 
%{\it Some remarks on topological full groups of  Cantor minimal systems II},
%to appear in Ergodic Theory Dynam. Systems.

%\bibitem{MatuiPLMS}{\sc H. Matui}, 
%{\it Homology and topological full groups of {\'e}tale groupoids on totally disconnected spaces},
%Proc. London Math. Soc. {\bf 104}(2012),  pp.\ 27--56.


%\bibitem{Matui2015}
%{\sc H. Matui}, 
%{\it Topological full groups of one-sided shifts of finite type},
%J. Reine Angew. Math. {\bf 705}(2015), pp. \ 35--84.

%
%\bibitem{MPT}
%{\sc P. S. Muhly, D.  Pask and M. Tomforde},
%{\it Strong shift equivalence of $C^*$-correspondences},
%%Israel J. Math. {\bf 167} (2008), pp. \ 315--346. 

%\bibitem{Parry}{\sc W. Parry},
%{\it Intrinsic Markov chains},
%Trans. Amer. Math. Soc. {\bf 112}(1964), pp.\ 55-66.
%
%\bibitem{PP}
%{\sc W. Parry and M. Pollicott},
%{\it Zeta functions and the periodic orbit structure of hyperbolic dynamics},
%Ast{\'e}risque {\bf 187-188}(1990).



\bibitem{PS}{\sc W. Parry and D. Sullivan},
{\it A topological invariant for flows on one-dimensional spaces},
Topology {\bf 14}(1975), pp.\ 297--299.

%\bibitem{Pat}{\sc A. L. T. Paterson},
%{\it Groupoids, inverse semigroups, and their operator algebras},
%Progress in Mathematics 170, Birkh{\"a}user, Boston, Basel, Berlin (1998).


%\bibitem{Po}{\sc Y. T. Poon},
%{\it A K-theoretic invariant for dynamical systems}, 
%Trans.\  Amer.\ Math.\ Soc.\ 
%{\bf 311}(1989), pp.\ 513--533.


%\bibitem{Phillips}
%{\sc N. C. Phillips},
%{\it A classification theorem for nuclear purely infinite simple $C^*$-algebras}, 
% Doc. Math. {\bf 5}(2000), pp. \ 49--114.


\bibitem{Putnam1}{\sc I. F. Putnam},
{\it  $C^*$-algebras from Smale spaces},
Can.\ J.\ Math.\ {\bf 48}(1996), pp.\ 175--195.

\bibitem{Putnam2}{\sc I. F. Putnam},
{\it  Hyperbolic systems and generalized Cuntz--Krieger algebras},
Lecture Notes, Summer School in Operator Algebras, Odense August 1996.


%\bibitem{Putnam3}{\sc I. F. Putnam},
%{\it  Functoriality of the $C^*$-algebras associated with hyperbolic dynamical systems},
%J.\ London Math.\ Soc.\ {\bf 62}(2000), pp.\ 873--884.


\bibitem{Putnam4}{\sc I. F. Putnam},
{\it  A homology theory for  Smale spaces},
Memoirs of Amer. Math. Soc. {\bf 232}(2014), No. 1094.



\bibitem{PutSp}{\sc I. F. Putnam and J. Spielberg},
{\it  The structure of $C^*$-algebras associated with hyperbolic dynamical systems},
J. Func. Anal. {\bf 163}(1999), pp. \. 279--299.



%\bibitem{Put}{\sc I. F. Putnam},
%{\it The $C^*$-algebras associated with minimal homeomorphisms of the Cantor set},
%Pacific  J.\ Math.\ {\bf 136}(1989), pp.\ 329--353.



%\bibitem{Ph}
%{\sc N. C. Phillips},
%{\it A classification theorem for nuclear  purely infinite simple
% $C^*$-algebras}, Doc.\ Math.\ {\bf 5}(2000), pp.\ 49--114.



%\bibitem{RaeWill}  
%{\sc I. Raeburn and D. P. Williams},
%{\it Morita equivalence and continuous-trace $C^*$-algebras},
%Mathematical Surveys and Monographs, vol(60)
%Amer. Math. Soc. (1998).



\bibitem{Renault1}{\sc J. Renault},
{\it A groupoid approach to $C^*$-algebras},
Lecture Notes in Math.  793, 
Springer-Verlag, Berlin, Heidelberg and New York (1980).


\bibitem{Renault2}{\sc J. Renault},
{\it  Cartan subalgebras in $C^*$-algebras},
Irish Math. Soc. Bull. 
{\bf 61}(2008), pp.\ 29--63. 




\bibitem{Renault3}{\sc J. Renault},
{\it Examples of masas in $C^*$-algebras},
 Operator structures and dynamical systems, pp.\ 259--265, 
 Contemp. Math., 503, Amer. Math. Soc., Providence, RI, 2009. 

%\bibitem{Rieffel1}
%{\sc M. A. Rieffel},
%{\it Induced representations of $C^*$-algebras},
%Adv.\ in Math.\
%{\bf 13}(1974), pp. 176--257.


%\bibitem{Rieffel2}
%{\sc M. A. Rieffel},
%{\it  Morita equivalence for  $C^*$-algebras and  $W^*$-algebras},
%J.\ Pure Appl.\ Algebra {\bf 5}(1974), pp.\ 51--96.  




%\bibitem{Ro}{\sc M. R{\o}rdam},
%{\it Classification of Cuntz-Krieger algebras},
 %K-theory {\bf 9}(1995), pp.\  31--58.

%\bibitem{Ro2}{\sc M. R{\o}rdam},
%{\it Classification of purely infinite simple $C^*$-algebras I}, 
%J. Func. Anal.\ {\bf 131}(1995), pp.\ 415--458.




\bibitem{Rosenberg}
{\sc J. Rosenberg},
{\it  Appendix to O. Bratteli's paper on "Crossed products of UHF algebras},
Duke Math. \. J. 
{\bf 46}(1979), pp. \ 25--26



\bibitem{RS}
{\sc J. Rosenberg and C. Schochet},
{\it The K{\"u}nneth theorem 
and the universal coefficient theorem for Kasparov's generalized K-functor}, 
Duke Math.\ J.\
{\bf 55}(1987),  pp.\ 431--474.



\bibitem{Ruelle1}{\sc D. Ruelle},
{\it Thermodynamic formalism}, Addison-Wesley, Reading (Mass.) (1978).

\bibitem{Ruelle2}{\sc D. Ruelle},
{\it Non-commutative algebras for hyperbolic diffeomorphisms},
Invent. Math.  {\bf 93}(1988), pp.\ 1--13.




%\bibitem{Ruelle3}{\sc D. Ruelle},
%{\it Dynamical zeta functions and transfer operators},
%Notice  Amer.\ Math.\ Soc. {\bf 49}(2002), pp.\ 175--193.



%\bibitem{Ruelle3}{\sc D. Ruelle},
%{\it Dynamical zeta functions and transfer operators},
%Notice  Amer.\ Math.\ Soc. {\bf 49}(2002), pp.\ 175--193.



\bibitem{Seneta}{\sc E. Seneta},
{\it Non-negative Matrices and Markov Chains},
Second Edition, Springer series in statistics, 
Springer-Verlag, Berlin, Heidelberg and New York (1981).


\bibitem{Smale}{\sc S. Smale},
{\it Differentiable dynamical systems},
Bull. Amer. Math. Soc. {\bf 73}(1967), pp. \ 747--817.



%\bibitem{Thomsen}
%{\sc K. Thomsen},
%{\it $C^*$-algebras of homoclinic and heteroclinic structure in expansive dynamics},
%Memoirs of Amer. Math. Soc. {\bf 206}(2010), No 970.


%\bibitem{Tomforde}{\sc M. Tomforde},
%{\it Strong shift equivalence in the $C^*$-algebraic setting:
% graphs and $C^*$-correspondences},  
%Operator theory, Operator Algebras, and Applications,  221--230, Contemp. Math., {\bf 414}, Amer. Math. Soc., Providence, RI, 2006.


%\bibitem{TomfordProblem}{\sc M. Tomforde},
%{\it The Graph Algebra Problem Page: List of Open Problems},
%http://www.math.uh.edu/tomforde/GraphAlgebraProblems/ListOfProblems.html.



%\bibitem{Walter}
%{\sc P. Walters},
%{\it An Introduction to Ergodic Theory},
%Springer Graduate Texts in Math {\bf 79}, New York, Heidelberg, 1982.





%%%%%%%%%%%%%%%%
%\bibitem{We} {\sc B. Weiss},
%{\it Subshifts of finite type and sofic systems}, 
%Monats.\ Math.\ {\bf 77},
%(1973), pp.\ 462--474.
%%%%%%%%%%%%%


%%%%%%%%%%%%%%%%%%%%%%%%%%%%
\bibitem{Williams} {\sc R. F. Williams},
{\it Classification of subshifts of finite type}, 
 Ann.\ Math.\  {\bf 98}(1973), pp.\ 120--153.
 erratum, Ann.\ Math.\
{\bf 99}(1974), pp.\ 380--381.
%%%%%%%%%%%%%%%%%%%%%%%%%%%%%%%%%%%%%%%%

\end{thebibliography}
\end{document}